\documentclass[11pt,a4paper]{article}

\usepackage[T1]{fontenc}
\usepackage[utf8]{inputenc}
\usepackage{lmodern}
\usepackage[a4paper,margin=2.6cm,headheight=14pt]{geometry}
\usepackage{microtype}
\usepackage{tikz}
\usetikzlibrary{calc}

\usepackage{amsmath,amssymb,amsthm,mathtools}
\usepackage{mathrsfs}
\usepackage{bm}

\usepackage{enumitem}
\usepackage{booktabs}
\usepackage{array}
\usepackage{tikz}

\usetikzlibrary{
  arrows.meta,
  positioning,
  calc,
  shapes.geometric,
  decorations.pathreplacing
}

\setlist[itemize]{
  leftmargin=2.1em,
  itemsep=0.25em,
  topsep=0.35em
}

\setlist[enumerate]{
  leftmargin=2.4em,
  itemsep=0.25em,
  topsep=0.35em
}

\usepackage{fancyhdr}

\usepackage{xcolor}

\definecolor{linkblue}{RGB}{30,70,120}

\usepackage[
  colorlinks=true,
  linkcolor=linkblue,
  citecolor=linkblue,
  urlcolor=linkblue,
  bookmarksnumbered=true,
  pdfauthor={Chong Shen, Xiaoyong Xi, Weng Kin Ho, Dongsheng Zhao},
  pdftitle={
    Scott Function Spaces under One-Sided FS Assumptions:
    Counterexamples, Positive Results, and New Directions
  }
]{hyperref}

\usepackage[nameinlink,capitalise,noabbrev]{cleveref}

\newtheorem{theorem}{Theorem}[section]
\newtheorem{proposition}[theorem]{Proposition}
\newtheorem{lemma}[theorem]{Lemma}
\newtheorem{corollary}[theorem]{Corollary}
\newtheorem{conjecture}[theorem]{Conjecture}

\theoremstyle{definition}
\newtheorem{definition}[theorem]{Definition}
\newtheorem{construction}[theorem]{Construction}

\newtheorem{question}[theorem]{Question}

\theoremstyle{remark}
\newtheorem{remark}[theorem]{Remark}

\newcommand{\id}{\operatorname{id}}

\newcommand{\FinDir}{\operatorname{FinDir}}
\newcommand{\Idl}{\operatorname{Idl}}
\newcommand{\rank}{\operatorname{rk}}

\newcommand{\below}{\mathord{\ll}}

\newcommand{\dua}{\mathord{\rotatebox[origin=c]{90}{$\twoheadrightarrow$}}}
\newcommand{\dda}{\mathord{\rotatebox[origin=c]{90}{$\twoheadleftarrow$}}}

\newcommand{\waybelowset}[1]{\dda#1}
\newcommand{\wayaboveset}[1]{\dua#1}

\newcommand{\bottom}{\bot}
\newcommand{\topel}{\top}

\newcommand{\Disk}{\mathsf{Disk}}

\title{
  \bfseries
  Scott Function Spaces under One-Sided FS Assumptions:\\
  Counterexamples, Positive Results, and New Directions
}

\author{
  Chong Shen\thanks{School of Science, Beijing University of Post and Telecommunications, Beijing, China, }
  \and
  Weng Kin Ho\thanks{Mathematics and Mathematics Education, National Institute of Education, Nanyang Technological University, Singapore, wengkin.ho@nie.edu.sg}
  \and
  Xiaoyong Xi\thanks{School of Mathematics and Statistics, Yancheng Teachers University, Jiangsu, Yancheng, China, }
  \and
  Dongsheng Zhao\thanks{Mathematics and Mathematics Education, National Institute of Education, Nanyang Technological University, Singapore, dongsheng.zhao@nie.edu.sg}
}

\date{}

\begin{document}

\maketitle

\begin{abstract}
The class of FS-domains is known to be closed under Scott function
spaces when both the source and target are FS-domains. This paper
investigates what remains true under one-sided FS assumptions, with
particular emphasis on the role of Plotkin's tie. We establish two
complementary continuity theorems. First, whenever \(X\) is an
FS-domain, the Scott function space \([X\to T]\) is a continuous dcpo.
The proof introduces finite-layer truncation maps on Plotkin's tie,
which generate directed families of way-below approximants below every
Scott-continuous map. Secondly, whenever \(L\) is an FS-domain, the
Scott function space \([T\to L]\) is again a continuous dcpo. Here the
argument is based on finitely separating approximate identities,
together with a finite-control analysis of the two-branch order
structure of Plotkin's tie. These two approximation mechanisms are
conceptually different but both produce the directed families of
way-below approximants required for continuity.

To determine the limits of these positive results, we consider the
Lawson closed-disk domain. Although \(\Disk^{\top}\) is an FS-domain,
the Scott function space \([\Disk^{\top}\to T]\) is shown to be
continuous but not itself an FS-domain. This establishes that
preservation of continuity is strictly weaker than preservation of the
FS property. The paper concludes by identifying the boundaries of the
present methods and proposing a unified approximation principle that
may provide a general characterization of continuity for Scott function
spaces.
\end{abstract}

\medskip


\medskip



\section{Introduction}
\label{sec:introduction}

Scott function spaces are fundamental to domain theory because they
provide the order-theoretic setting in which higher-order computation
can be interpreted. Given dcpos \(D\) and \(E\), the pointwise ordered
dcpo \([D\to E]\) of Scott-continuous maps represents computations that
take elements of \(D\) as inputs and return elements of \(E\) as
outputs. Closure under Scott function spaces is therefore essential in
the construction of cartesian closed categories of domains and, more
generally, in the solution of recursive domain equations arising in
denotational semantics. However, Scott function-space formation does
not preserve continuity in general. Consequently, a central theme in
domain theory has been to identify subclasses whose approximation
structure is sufficiently stable under the formation of exponentials
\cite{AbramskyJung1994,GoubaultLarrecq2013}.

The search for cartesian closed categories of domains grew out of
Scott's order-theoretic approach to the semantics of computation, in
which continuous domains provide mathematical models for recursively
defined higher-order programs \cite{Scott1970,Scott1972}. This
programme was further advanced by Plotkin's denotational semantics for
the programming language PCF, leading to the introduction of the class
of SFP-domains (Sequentially Generated from Finite Posets), now more
commonly known as bifinite domains \cite{Plotkin1977}. Their rich
finite approximation structure ensures that the Scott function space of
two bifinite domains is again bifinite, making bifinite domains one of
the first successful cartesian closed categories for higher-order
denotational semantics. Their success naturally prompted the search for
larger classes of domains enjoying the same closure property. This
programme culminated in Jung's celebrated classification theorem, which
identifies precisely two maximal cartesian closed full subcategories of
continuous domains, namely the L-domains and the FS-domains. Among
these, FS-domains substantially extend the class of bifinite domains by
replacing finite-image deflations with the more flexible notion of
finitely separated approximations to the identity, while still
satisfying the fundamental closure theorem that, whenever both \(D\)
and \(E\) are FS-domains, the Scott function space \([D\to E]\) is
again an FS-domain \cite{AbramskyJung1994,Jung1990}.

The two-sided closure theorem naturally raises the question of how much
of this conclusion survives when only one of the two domains is assumed
to be an FS-domain. In this direction, Abramsky and Jung proposed the
following statement \cite[Proposition~4.2.10]{AbramskyJung1994}:

\begin{quote}
\emph{If \(D\) is an FS-domain and \(E\) is pointed and continuous,
then \([D\to E]\) is continuous.}
\end{quote}

Immediately following this proposition, the authors remarked that its
proof was ``not only trickier'' but ``as yet unknown'', since the
proposed argument depended on the unresolved question of whether every
FS-domain is a Scott-continuous retract of a bifinite domain,
equivalently, whether every FS-domain is an RB-domain. Thus, despite
the cartesian closedness of the category of FS-domains, the relaxation
from an FS target to an arbitrary pointed continuous target remained
open for more than three decades.

The present paper shows that this assertion is, in fact, false. A
recent breakthrough of Chen, Kou and Lyu established that the class of
FS-domains is strictly larger than that of RB-domains, thereby removing
the key assumption on which the proposed proof strategy depended.
Building upon this development, we construct an explicit counterexample
consisting of an FS-domain \(D\) and a carefully chosen algebraic
domain \(L\) for which the Scott function space \([D\to L]\) fails to
be continuous. This disproves Proposition~4.2.10 of Abramsky and Jung,
thereby settling the one-sided closure problem in the negative.
Nevertheless, this negative answer is only part of the story
\cite{ChenKouLyu2026}.

The counterexample reveals that continuity of Scott function spaces
under one-sided FS assumptions depends in a subtle way on the
interaction between the source and target domains. To investigate this
phenomenon, we turn to Plotkin's tie, a classical non-bifinite
algebraic domain introduced by Plotkin in his study of PCF. We prove
that, for every FS-domain \(D\), both Scott function spaces
\([D\to T]\) and \([T\to D]\) are continuous. These two positive
results are established by fundamentally different approximation
mechanisms, reflecting the distinct roles played by the source and the
target in Scott function-space formation. Together with the
counterexample, they completely resolve the one-sided FS continuity
problem posed implicitly by Proposition~4.2.10: the general statement
is false, yet two natural one-sided classes continue to preserve
continuity.

The remainder of this paper is organized as follows. Section~2 reviews
the necessary background on Scott function spaces, bifinite domains,
RB-domains, and FS-domains. Section~3 introduces the two classical
domains that play central roles throughout the paper, namely Plotkin's
tie and the closed-disk domain. In Section~4, we construct a
counterexample showing that Scott function spaces need not be
continuous under one-sided FS assumptions, thereby disproving
Proposition~4.2.10 of Abramsky and Jung. Sections~5 and~6 establish two
complementary positive results by proving the continuity of the Scott
function spaces \([D\to T]\) and \([T\to D]\), respectively, for every
FS-domain \(D\). Section~7 contrasts the behaviours of the closed-disk
domain and Plotkin's tie, highlighting how the interaction between the
approximation structures of the source and target determines continuity
of Scott function spaces under one-sided FS assumptions. Finally,
Section~8 discusses the consequences of these results, identifies the
boundaries of the present methods, and proposes several open problems
arising from this work.

\section{Preliminaries on Scott Function Spaces and FS-Domains}
\label{sec:preliminaries}

This section briefly reviews the notions and results from domain theory that
will be used throughout the paper. Standard references include
\cite{AbramskyJung1994,GierzEtAl2003,GoubaultLarrecq2013}.

\subsection{Continuous dcpos and Scott function spaces}
\label{subsec:scott-function-spaces}

We briefly recall the notions of continuous dcpos and Scott function spaces
used throughout the paper. Standard references include
\cite{AbramskyJung1994,GierzEtAl2003,GoubaultLarrecq2013}.

A \emph{directed complete partial order} (dcpo) is a partially ordered set
\(D\) in which every directed subset has a supremum. A map
\(f:D\to E\) between dcpos is \emph{Scott-continuous} if it is monotone and
preserves the suprema of directed subsets.

Let \(D\) be a dcpo. An element \(a\in D\) is said to be \emph{way below}
\(x\in D\), written \(a\below x\), if for every directed subset
\(A\subseteq D\) satisfying \(x\le\bigvee A\), there exists \(y\in A\) such
that \(a\le y\).

Following standard notation, we write
\[
\waybelowset{x}
=
\{a\in D:a\below x\},
\qquad
\wayaboveset{x}
=
\{y\in D:x\below y\},
\]
for the way-below and way-above sets of \(x\), respectively. A dcpo \(D\) is
said to be \emph{continuous} if, for every \(x\in D\), the set
\(\waybelowset{x}\) is directed and \(x=\bigvee\waybelowset{x}\).

Given dcpos \(D\) and \(E\), we denote by \([D\to E]\) the dcpo of all
Scott-continuous maps from \(D\) to \(E\), ordered pointwise; that is,
\(f\le g\) whenever \(f(x)\le g(x)\) for every \(x\in D\). Whenever \(D\) and \(E\) are dcpos, the Scott function space
\([D\to E]\) is again a dcpo, with directed suprema computed
pointwise.


\subsection{Bifinite domains}
\label{subsec:bifinite}

Bifinite domains, originally introduced by Plotkin as SFP-domains
(Sequentially Generated from Finite Posets), form one of the earliest
cartesian closed categories in domain theory. They admit several equivalent
characterizations. Among these, we adopt the formulation in terms of
finite-image deflations, as it provides the natural point of departure for the
subsequent notions of RB-domains and FS-domains.

A \emph{deflation} on a dcpo \(D\) is a Scott-continuous map
\(f:D\to D\) satisfying \(f\circ f=f\) and \(f(x)\le x\) for every
\(x\in D\). A deflation is said to be \emph{finite-image} if its image is
finite. A continuous dcpo \(D\) is said to be \emph{bifinite} if there exists
an increasing sequence \((f_n)_{n\ge1}\) of finite-image deflations such that
\(x=\bigvee_{n\ge1}f_n(x)\) for every \(x\in D\).

One of the fundamental results of domain theory asserts that the class of
bifinite domains is cartesian closed; that is, if \(D\) and \(E\) are
bifinite domains, then the Scott function space \([D\to E]\) is again a
bifinite domain
\cite{AbramskyJung1994,GierzEtAl2003}.


\subsection{RB-domains}
\label{subsec:rb}

RB-domains were introduced by Lawson as a natural generalization of bifinite
domains, replacing increasing sequences of finite-image deflations by directed
families.

A continuous dcpo \(D\) is said to be an \emph{RB-domain} if there exists a
directed family \((f_i)_{i\in I}\) of finite-image deflations such that
\(x=\bigvee_{i\in I}f_i(x)\) for every \(x\in D\).

Every bifinite domain is an RB-domain, and every RB-domain is an FS-domain
\cite{AbramskyJung1994,Lawson2008}.


\subsection{FS-domains}
\label{subsec:fs}

FS-domains were introduced by Jung as a further generalization of
RB-domains, replacing finite-image deflations by the broader class of
Scott-continuous maps that are finitely separated from the identity.

Let \(D\) be a dcpo. A Scott-continuous map \(f:D\to D\) is said to be
\emph{finitely separated from the identity} if there exists a finite subset
\(M\subseteq D\) such that, for every \(x\in D\), there exists
\(m\in M\) satisfying \(f(x)\le m\le x\).

A continuous dcpo \(D\) is said to be an \emph{FS-domain} if there exists a
directed family \((f_i)_{i\in I}\) of Scott-continuous maps, each finitely
separated from the identity, such that
\(x=\bigvee_{i\in I}f_i(x)\) for every \(x\in D\).

Jung proved that the resulting class of FS-domains is cartesian closed; that
is, if \(D\) and \(E\) are FS-domains, then the Scott function space
\([D\to E]\) is again an FS-domain
\cite{AbramskyJung1994,Jung1990}.


\subsection{Ideal completion}
\label{subsec:ideal-completion}

Ideal completion provides a canonical method of constructing algebraic
domains from arbitrary posets. It will play a central role in the
construction of the target domain in Section~\ref{subsec:algebraic-target}.

Recall that an \emph{ideal} of a poset \(P\) is a nonempty directed
lower subset of \(P\). We denote by \(\Idl(P)\) the collection of all
ideals of \(P\), ordered by inclusion.

The following classical result is fundamental.

\begin{proposition}[Ideal Completion Theorem]
\label{prop:ideal-completion}
For every poset \(P\), the ideal completion \(\Idl(P)\) is an algebraic
dcpo. Its compact elements are precisely the principal ideals
\[
\downarrow_P p
=
\{q\in P:q\le p\},
\qquad p\in P.
\]
Moreover, directed suprema in \(\Idl(P)\) are given by unions.
\end{proposition}

This is classical; see
\cite[Proposition~2.2.22]{AbramskyJung1994}.


\subsection{Retracts}
\label{subsec:retracts}

Retracts provide one of the principal mechanisms for transferring
approximation properties between domains. A continuous dcpo \(D\) is
said to be a \emph{Scott-continuous retract} of a continuous dcpo
\(E\) if there exist Scott-continuous maps
\[
s:D\to E,
\qquad
r:E\to D,
\]
such that
\[
r\circ s=\id_D.
\]
The map \(s\) is called a \emph{section}, while
\(r\) is called a \emph{retraction}.

Retracts preserve many important approximation properties. In
particular, every Scott-continuous retract of a bifinite domain is an
RB-domain \cite{AbramskyJung1994,Jung1990}. Moreover, both the classes
of RB-domains and FS-domains are closed under Scott-continuous
retracts. This permanence property will play a crucial role in
Section~\ref{sec:disk-tie-contrast}.

\section{Two Classical Boundary Domains}
\label{sec:boundary-domains}

This section recalls two classical domains that play central roles throughout
the paper. Plotkin's tie serves as the distinguished algebraic domain in our
two positive results, while the closed-disc domain provides the starting point
for the construction of our counterexample. Their contrasting approximation
structures ultimately account for the different behaviours of Scott function
spaces established in the subsequent sections.

\subsection{Plotkin's Tie}
\label{subsec:plotkin-tie}

Plotkin's tie was introduced by Plotkin in his study of the denotational
semantics of the programming language PCF and has since occupied a
distinguished position in domain theory. Despite its remarkably simple order
structure, it exhibits several striking properties. In particular, it is an
algebraic domain that is Lawson compact and coherent, yet fails to be
bifinite. These features make Plotkin's tie an ideal testing ground for
understanding the behaviour of Scott function spaces.

Figure~\ref{fig:plotkin-tie} depicts the Hasse diagram of Plotkin's tie. The
infinite family of crossing cover relations gives rise to the characteristic
``tie'' shape from which the domain derives its name. This deceptively simple
structure is responsible for many of its remarkable order-theoretic and
topological properties.

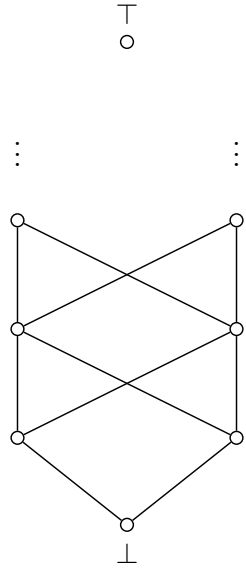
\begin{figure}[ht]
\centering

\begin{tikzpicture}[
    x=1.45cm,
    y=1.15cm,
    ordernode/.style={
        circle,
        draw,
        fill=white,
        inner sep=0pt,
        minimum size=4.8pt,
        line width=0.55pt
    },
    orderedge/.style={
        line width=0.55pt
    }
]

\node[ordernode] (b) at (0,0) {};

\node[ordernode] (l1) at (-1,1) {};
\node[ordernode] (r1) at ( 1,1) {};

\node[ordernode] (l2) at (-1,2.25) {};
\node[ordernode] (r2) at ( 1,2.25) {};

\node[ordernode] (l3) at (-1,3.50) {};
\node[ordernode] (r3) at ( 1,3.50) {};

\draw[orderedge] (b) -- (l1);
\draw[orderedge] (b) -- (r1);

\draw[orderedge] (l1) -- (l2);
\draw[orderedge] (r1) -- (r2);
\draw[orderedge] (l1) -- (r2);
\draw[orderedge] (r1) -- (l2);

\draw[orderedge] (l2) -- (l3);
\draw[orderedge] (r2) -- (r3);
\draw[orderedge] (l2) -- (r3);
\draw[orderedge] (r2) -- (l3);

\node at (-1,4.35) {\(\vdots\)};
\node at ( 1,4.35) {\(\vdots\)};

\node[ordernode] (t) at (0,5.55) {};

\node[below=3pt] at (b) {\(\bottom\)};
\node[above=3pt] at (t) {\(\topel\)};

\end{tikzpicture}
\caption{Plotkin's tie \(T\).}
\label{fig:plotkin-tie}
\end{figure}

We now give a formal description of Plotkin's tie.

\begin{definition}
\label{def:plotkin-tie}
Plotkin's tie is the poset
\[
T=\{\bottom,\topel\}\cup\{a_n,b_n:n\geq 1\}.
\]
The order on \(T\) is determined by
\[
\bottom\le x\le\topel
\qquad(x\in T),
\]
together with
\[
u_m\le v_n
\Longleftrightarrow
m<n
\text{ or }
(m=n \text{ and }u_m=v_n).
\]
Thus, \(\{a_n,b_n\}\) is a two-element antichain at each level \(n\), and
each element at level \(n\) lies below both elements at every higher level.
\end{definition}

\begin{proposition}
\label{prop:plotkin-tie-properties}
Plotkin's tie \(T\) satisfies the following properties.

\begin{enumerate}
\item
\(T\) is algebraic.

\item
\(T\) is Lawson compact.

\item
\(T\) is coherent.

\item
Although \(T\) is algebraic, it is not bifinite; whence it is not an FS domain.
\end{enumerate}
\end{proposition}

These properties are classical; see
\cite{AbramskyJung1994,GierzEtAl2003,GoubaultLarrecq2013}.

While Plotkin's tie serves as the algebraic domain underlying our two positive
results, the counterexample developed later in this paper arises from a rather
different source, namely the closed-disc domain. We now turn to this second
classical example.


\subsection{The Closed-Disk Domain}
\label{subsec:closed-disc}

The planar closed-disk domain was suggested by Jimmie Lawson and first
recorded by Jung in his classification of continuous domains. Its elements are
the closed disks in the Euclidean plane, together with the whole plane as a
least element, ordered by reverse inclusion. The example was introduced as an
FS-domain with a particularly transparent geometric structure and was already
regarded as a plausible candidate for an FS-domain that might fail to be an
RB-domain. Lawson later placed the construction in a more general setting by
showing that domains of closed formal balls over suitable metric spaces are
FS-domains; the planar closed-disk domain is a special case of this result
\cite{Jung1990,Lawson2008}.

For many years, however, it remained unknown whether the planar closed-disk
domain was an RB-domain. Chen, Kou and Lyu have recently proved that it is not,
thereby confirming Lawson's proposed candidate and establishing that the class
of FS-domains is strictly larger than the class of RB-domains
\cite{ChenKouLyu2026}. Consequently, the closed-disk domain provides the
geometric foundation for the counterexample developed in the next section.

Figure~\ref{fig:closed-disk} illustrates the geometric intuition behind the
closed-disk domain. The disks are ordered by reverse inclusion, so that a
smaller disk represents a larger element in the domain order.

\begin{figure}[ht]
\centering
\begin{tikzpicture}[scale=1]

\draw[line width=0.7pt] (0,0) circle (2);
\draw[line width=0.7pt] (0.2,0.1) circle (1.35);
\draw[line width=0.7pt] (0.35,0.15) circle (0.7);

\fill (0,0) circle (1.1pt);
\fill (0.2,0.1) circle (1.1pt);
\fill (0.35,0.15) circle (1.1pt);

\node at (2.45,1.15) {\(B_1\)};
\node at (1.65,-1.15) {\(B_2\)};
\node at (0.95,-0.35) {\(B_3\)};

\node at (0,-2.55) {\(B_1\supseteq B_2\supseteq B_3\)};
\node at (0,-3.05) {\(B_1\leq B_2\leq B_3\) in the domain order};

\end{tikzpicture}

\caption{Nested closed disks illustrating the reverse-inclusion order on the
closed-disk domain.}
\label{fig:closed-disk}
\end{figure}
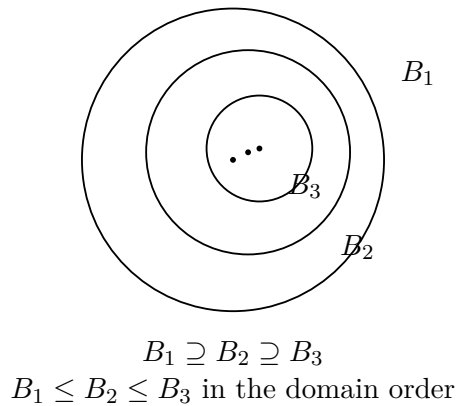

\begin{definition}
\label{def:closed-disk}

For \(z\in\mathbb{R}^{2}\) and \(r\geq 0\), let
\[
B(z,r)=\{w\in\mathbb{R}^{2}:\lVert w-z\rVert\leq r\}.
\]
The \emph{planar closed-disk domain}, denoted by \(\Disk\), consists of all
closed Euclidean disks \(B(z,r)\), together with the whole plane
\(\mathbb{R}^{2}\), ordered by reverse inclusion. Thus, for
\(B(z,r),B(w,s)\in\Disk\),
\[
B(z,r)\leq B(w,s)
\quad\Longleftrightarrow\quad
B(z,r)\supseteq B(w,s).
\]

The whole plane \(\mathbb{R}^{2}\) is the least element of \(\Disk\), while
the disks of radius \(0\), equivalently the singleton disks, are precisely
its maximal elements.
\end{definition}

\begin{proposition}
\label{prop:closed-disk-properties}
The planar closed-disk domain \(\Disk\) satisfies the following properties.

\begin{enumerate}
\item
\(\Disk\) is a continuous domain.

\item
\(\Disk\) is an FS-domain.

\item
\(\Disk\) is not an RB-domain.
\end{enumerate}
\end{proposition}

Properties (i) and (ii) are due to Lawson, while property (iii) was
established by Chen, Kou and Lyu
\cite{Lawson2008,ChenKouLyu2026}.

Unlike Plotkin's tie, the closed-disk domain itself plays no role in the positive continuity theorems; rather, it serves as the FS-domain from which the counterexample of Section 4 is constructed.

\section{Failure of Continuity for an FS Source and an Algebraic Target}
\label{sec:counterexample}

In this section, we establish the principal negative result of the paper by
constructing an explicit counterexample to Proposition~4.2.10 of Abramsky and
Jung. Starting from the planar closed-disk domain, we adjoin a compact
greatest element, construct a suitable algebraic target, and exhibit a
Scott-continuous retraction. We then analyse the finite-image approximations
of the corresponding section, culminating in a proof that the associated Scott
function space fails to be continuous.


\subsection{Adjoining a Compact Greatest Element}
\label{subsec:compact-top}

Throughout this section, let \(D\) be a pointed, continuous, non-RB
FS-domain. Our first step is to adjoin a new compact greatest element,
thereby obtaining another pointed domain that retains the approximation
properties of \(D\). This simple construction provides the source domain
used throughout the remainder of the counterexample.

\begin{construction}
\label{construction:compact-top}
Adjoin a new element \(\topel\notin D\), declare \(d\le\topel\) for every
\(d\in D\), and denote the resulting poset by
\[
X=D^{\top}=D\cup\{\topel\}.
\]
\end{construction}

\begin{lemma}
\label{lemma:compact-top-continuous}
The poset \(X=D^{\top}\) is a continuous dcpo, and the newly adjoined
greatest element \(\topel\) is compact.
\end{lemma}

\begin{proof}
Let \(A\subseteq X\) be directed. If \(\topel\in A\), then
\(\sup A=\topel\). Otherwise \(A\subseteq D\), so \(A\) has a supremum in
\(D\), which is also the supremum of \(A\) in \(X\). Hence \(X\) is a
dcpo.

To prove that \(\topel\) is compact, suppose that
\(\topel\le\sup A\). Since \(\topel\) is the greatest element of \(X\),
necessarily \(\sup A=\topel\). If \(\topel\notin A\), then
\(A\subseteq D\), and the previous paragraph implies that
\(\sup A\in D\), a contradiction. Thus \(\topel\in A\), proving
\(\topel\below\topel\).

Now let \(x\in D\), and suppose that \(a\below_D x\). Given a directed
subset \(A\subseteq X\) with \(x\le\sup A\), either \(\topel\in A\), in
which case \(a\le\topel\in A\), or \(A\subseteq D\), in which case
\(a\below_D x\) yields \(a\le d\) for some \(d\in A\). Hence
\(a\below_X x\), and therefore
\(\waybelowset[_D]{x}\subseteq\waybelowset[_X]{x}\).

Since \(D\) is continuous, \(\waybelowset[_D]{x}\) is directed and has
supremum \(x\); therefore \(x\) is also the directed supremum in \(X\) of
elements way below it. The new greatest element satisfies
\(\topel\below\topel\), so it is the supremum of the directed set
\(\{\topel\}\). Hence every element of \(X\) is the directed supremum of
elements way below it, and therefore \(X\) is continuous.
\end{proof}

\begin{corollary}
\label{cor:everything-below-top}
For every \(x\in X\), one has \(x\below\topel\).
\end{corollary}

\begin{proof}
Since \(x\le\topel\below\topel\), the auxiliary property of the
way-below relation gives \(x\below\topel\).
\end{proof}

\begin{lemma}
\label{lemma:compact-top-fs}
If \(D\) is an FS-domain, then \(X=D^{\top}\) is also an FS-domain.
\end{lemma}

\begin{proof}
Let \((\delta_i)_{i\in I}\) be a directed approximate identity on \(D\),
where each \(\delta_i\) is finitely separated from the identity by a
finite set \(M_i\). Extend each \(\delta_i\) to a map
\(\widehat{\delta_i}:X\to X\) by
\[
\widehat{\delta_i}(x)=
\begin{cases}
\delta_i(x), & x\in D,\\
\topel, & x=\topel.
\end{cases}
\]

The map \(\widehat{\delta_i}\) is clearly monotone. Let \(A\subseteq X\)
be directed. If \(\sup A\in D\), then \(A\subseteq D\), so Scott
continuity follows from that of \(\delta_i\). If \(\sup A=\topel\), then
Lemma~\ref{lemma:compact-top-continuous} shows that \(\topel\in A\), and
hence
\[
\widehat{\delta_i}(\sup A)=\topel
=\sup\widehat{\delta_i}[A].
\]
Thus \(\widehat{\delta_i}\) is Scott-continuous.

For \(x\in D\), the original separator \(M_i\) remains valid, while
\(\topel\) itself separates
\(\widehat{\delta_i}(\topel)=\topel\) from the identity. Hence
\(M_i\cup\{\topel\}\) finitely separates
\(\widehat{\delta_i}\) from \(\mathrm{id}_X\).

Finally, the family \((\widehat{\delta_i})_{i\in I}\) remains directed
and satisfies
\[
\sup_{i\in I}\widehat{\delta_i}
=\mathrm{id}_X.
\]
Therefore \(X\) is an FS-domain.
\end{proof}

\begin{lemma}
\label{lemma:compact-top-nonrb}
If \(D\) is not an RB-domain, then \(X=D^{\top}\) is not an RB-domain.
\end{lemma}

\begin{proof}
Suppose, to the contrary, that \(X\) is an RB-domain. Then there exists a
directed family of deflations \((q_j)_{j\in J}\) with
\[
\sup_{j\in J}q_j=\mathrm{id}_X.
\]

For every \(x\in D\), the inequality
\(q_j(x)\le x<\topel\) implies that \(q_j(x)\in D\). Thus each
restriction \(q_j|_D:D\to D\) is well defined. Since directed suprema of
subsets of \(D\) are unchanged when regarded in \(X\), each restriction
remains Scott-continuous. Moreover, it has finite image and satisfies
\(q_j|_D\le\mathrm{id}_D\).

The restricted family is still directed, and its pointwise supremum is
\(\mathrm{id}_D\). Hence it witnesses that \(D\) is an RB-domain,
contradicting the standing assumption. Therefore \(X=D^{\top}\) is not
an RB-domain.
\end{proof}

We have therefore obtained a pointed, continuous, non-RB FS-domain
\(X=D^{\top}\). In the next subsection we construct a pointed algebraic
domain \(L\) into which \(X\) admits a Scott-continuous retraction.

\subsection{Constructing an Algebraic Target}
\label{subsec:algebraic-target}

Having constructed the source domain \(X=D^{\top}\), we now turn to the
construction of the target domain. A natural first attempt would be to
take the ideal completion of a basis \(B\) of \(X\). However, this does
not yield the local finiteness property required later in the proof,
since principal ideals of \(B\) may be infinite.

Our solution is to replace individual basis elements by finite directed
subsets of \(B\). The resulting poset retains enough of the approximation
structure of \(X\), while its ideal completion is algebraic by the Ideal
Completion Theorem. More importantly, the compact elements of the ideal
completion have finite principal lower sets, giving precisely the local
finiteness property on which the counterexample ultimately depends.

We begin by explaining what is meant by a finite directed subset.

\paragraph{Finite directed subsets.}

Since the construction of the target domain is based on finite directed
subsets of a basis of \(X\), we first clarify what directedness means in
this setting. Recall that a subset \(F\subseteq X\) is directed if it is
nonempty and every pair of elements of \(F\) has a common upper bound
that already belongs to \(F\). Thus directedness is an intrinsic
property of the ordered set \(F\), rather than of the ambient dcpo
\(X\).

For finite subsets, directedness admits a particularly simple
characterization.

\begin{lemma}
\label{lemma:finite-directed-greatest}
A nonempty finite subset \(F\subseteq X\) is directed if and only if it
has a greatest element.
\end{lemma}

\begin{proof}
Suppose first that \(F\) is directed. Write
\(F=\{b_1,\ldots,b_n\}\). Since \(F\) is directed, there exists an
element of \(F\) above both \(b_1\) and \(b_2\). Combining this element
with \(b_3\) and continuing inductively, we obtain an element
\(c_F\in F\) satisfying \(b\le c_F\) for every \(b\in F\). Thus \(c_F\)
is the greatest element of \(F\).

Conversely, if \(F\) has a greatest element \(c_F\), then \(c_F\) is a
common upper bound in \(F\) for every pair of elements of \(F\). Hence
\(F\) is directed.
\end{proof}

By Lemma~\ref{lemma:finite-directed-greatest}, henceforth we shall
identify a finite directed subset with a nonempty finite subset having a
greatest element.

For example, \(\{\bot,b\}\) is directed for every \(b\in B\), with
greatest element \(b\). On the other hand,
\(\{\bot,b_1,b_2\}\) need not be directed if neither \(b_1\) nor
\(b_2\) is above the other. If \(c\in B\) satisfies
\(b_1,b_2\le c\), then \(\{\bot,b_1,b_2,c\}\) is directed, with
greatest element \(c\). In particular, since the adjoined element
\(\topel\) belongs to \(B\) and is the greatest element of \(X\),
adjoining \(\topel\) to any finite subset of \(B\) always produces a
finite directed subset.

\begin{construction}
\label{construction:finite-directed-poset}

Let \(B\) be a basis of \(X\) containing both \(\bot\) and
\(\topel\). Define
\[
P
:=
\FinDir_{\bot}(B)
=
\left\{
F\subseteq B:
\begin{array}{l}
F\text{ is finite and directed},\\
\bot\in F
\end{array}
\right\},
\]
ordered by set inclusion.

Thus the order on \(P\) is inclusion between finite subsets of \(B\),
rather than the original order on the elements of \(X\). The condition
\(\bot\in F\) ensures that every member of \(P\) is nonempty and, as we
shall see in the next lemma, gives \(P\) a least element.

Finally, define
\[
L:=\Idl(P).
\]
\end{construction}


\begin{lemma}
\label{lemma:P-directed}
The poset \(P\) has least element \(\{\bot\}\) and is directed.
\end{lemma}

\begin{proof}
The singleton \(\{\bot\}\) belongs to \(P\), and every \(F\in P\)
contains \(\bot\). Hence \(\{\bot\}\subseteq F\) for every \(F\in P\),
so \(\{\bot\}\) is the least element of \(P\).

Now let \(F,G\in P\). Set
\[
H:=F\cup G\cup\{\topel\}.
\]
Then \(H\) is a finite subset of \(B\), contains \(\bot\), and has
greatest element \(\topel\). By
Lemma~\ref{lemma:finite-directed-greatest}, \(H\) is directed, and hence
\(H\in P\). Since \(F,G\subseteq H\), the elements \(F\) and \(G\) have
a common upper bound in \(P\). Therefore \(P\) is directed.
\end{proof}


\begin{proposition}
\label{prop:algebraic-target}
The dcpo \(L=\Idl(P)\) is a pointed algebraic domain. Its compact
elements are precisely the principal ideals \(\downarrow_P F\), where
\(F\in P\). If \(B\) is countable, then \(L\) is
\(\omega\)-algebraic. Moreover, \(L\) has greatest element
\(\topel_L=P\).
\end{proposition}

\begin{proof}
By Proposition~\ref{prop:ideal-completion}, the ideal completion
\(L=\Idl(P)\) is an algebraic dcpo, and its compact elements are
precisely the principal ideals \(\downarrow_P F\), where \(F\in P\).

By Lemma~\ref{lemma:P-directed}, the poset \(P\) has least element
\(\{\bot\}\). Hence its principal ideal
\[
\downarrow_P\{\bot\}
=
\bigl\{\{\bot\}\bigr\}
\]
is the least element of \(L\). Therefore \(L\) is pointed.

If \(B\) is countable, then the collection of all finite subsets of
\(B\) is countable. Since \(P\) is a subcollection of this collection,
\(P\) is countable. Consequently,
\(\{\downarrow_P F:F\in P\}\) is a countable basis of compact elements
for \(L\), and hence \(L\) is \(\omega\)-algebraic.

Finally, Lemma~\ref{lemma:P-directed} shows that \(P\) is directed.
Moreover, \(P\) is nonempty and is trivially a lower subset of itself.
Thus \(P\) is an ideal of the poset \(P\), and hence an element of
\(L=\Idl(P)\). Since every ideal of \(P\) is contained in \(P\), this
ideal is the greatest element of \(L\). We therefore write
\(\topel_L=P\).
\end{proof}

\begin{lemma}[Finite principal lower sets below the top]
\label{lemma:finite-principal-lowersets}
If \(a\below\topel_L\), then the principal lower set
\[
\downarrow_L a
=
\{I\in L:I\subseteq a\}
\]
is finite.
\end{lemma}

\begin{proof}
By Proposition~\ref{prop:algebraic-target},
\(\topel_L=P\). Since \(P\) is directed, the family of principal ideals
\(\{\downarrow_PF:F\in P\}\) is directed in \(L\). By
Proposition~\ref{prop:ideal-completion}, its supremum is their union,
namely
\[
\bigcup_{F\in P}\downarrow_PF
=
P
=
\topel_L.
\]

Since \(a\below\topel_L\), there exists \(F\in P\) such that
\(a\subseteq\downarrow_PF\). Moreover,
\[
\downarrow_PF
=
\{E\in P:E\subseteq F\}.
\]
Because \(F\) is finite, it has only finitely many subsets. Hence
\(\downarrow_PF\) is finite.

Now let \(I\in\downarrow_La\). Then
\(I\subseteq a\subseteq\downarrow_PF\). Hence every element of
\(\downarrow_La\) is a subset of the fixed finite set
\(\downarrow_PF\). Therefore \(\downarrow_La\) is finite.
\end{proof}

The source domain \(X\) and the algebraic target domain \(L\) have now
been constructed. We next connect them by constructing
Scott-continuous maps \(s:X\to L\) and \(r:L\to X\), and prove that
\(r\circ s=\id_X\). This relationship will allow the local finiteness
of \(L\) established above to be transferred back to the source domain
\(X\), ultimately leading to the desired counterexample.

\subsection{A Scott-Continuous Retraction}
\label{subsec:retraction}

Having constructed the algebraic target domain \(L\), we now connect it
to the source domain \(X\). The section map records all finite directed
approximations to an element of \(X\) arising from the chosen basis,
while the retraction reconstructs an element of \(X\) by taking the
supremum of the basis elements belonging to an ideal. We shall prove
that both maps are Scott-continuous and satisfy \(r\circ s=\id_X\).


\begin{construction}
\label{construction:section-retraction}

For each \(x\in X\), let \(s(x)\) consist of all finite directed
subsets of the chosen basis whose elements approximate \(x\); that is,
\[
s(x)
:=
\{F\in P:F\subseteq B\cap\waybelowset{x}\}.
\]

For each ideal \(I\in L\), define \(r(I)\) to be the supremum of all
basis elements appearing in the members of \(I\):
\[
r(I)
:=
\bigvee_X\bigcup I.
\]
\end{construction}


\begin{lemma}
\label{lemma:sx-ideal}
For every \(x\in X\), the set \(s(x)\) is an ideal of \(P\).
Consequently, the assignment in
Construction~\ref{construction:section-retraction} defines a map
\(s:X\to L\).
\end{lemma}

\begin{proof}
First, \(\bot\below x\) and \(\bot\in B\), so
\(\{\bot\}\subseteq B\cap\waybelowset{x}\). Since
\(\{\bot\}\in P\), it follows that \(\{\bot\}\in s(x)\). Thus \(s(x)\)
is nonempty.

Next, suppose that \(F\in s(x)\) and \(E\in P\) with \(E\subseteq F\).
Then
\[
E\subseteq F\subseteq B\cap\waybelowset{x},
\]
so \(E\in s(x)\). Hence \(s(x)\) is a lower subset of \(P\).

Finally, let \(F,G\in s(x)\). Since \(B\) is a basis,
\(B\cap\waybelowset{x}\) is directed. As \(F\cup G\) is finite, there
exists \(c\in B\cap\waybelowset{x}\) such that \(b\le c\) for every
\(b\in F\cup G\). Set
\[
H:=F\cup G\cup\{c\}.
\]
Then \(H\) is finite, contains \(\bot\), and has greatest element \(c\).
By Lemma~\ref{lemma:finite-directed-greatest}, \(H\) is directed, so
\(H\in P\). Moreover,
\(F,G\subseteq H\subseteq B\cap\waybelowset{x}\), and hence
\(H\in s(x)\). Therefore \(s(x)\) is directed.

Thus \(s(x)\) is a nonempty directed lower subset of \(P\), and hence an
ideal.
\end{proof}


\begin{lemma}
\label{lemma:s-scott-continuous}
The map \(s:X\to L\) is Scott-continuous.
\end{lemma}

\begin{proof}
First, \(s\) is monotone. Indeed, if \(x\le y\), then the auxiliary
property of the way-below relation gives
\(\waybelowset{x}\subseteq\waybelowset{y}\). Hence
\(s(x)\subseteq s(y)\).

Now let \((x_i)_{i\in I}\) be a directed family in \(X\), and put
\(x=\bigvee_{i\in I}x_i\). By monotonicity,
\[
\bigcup_{i\in I}s(x_i)\subseteq s(x).
\]

For the reverse inclusion, let \(F\in s(x)\). Then
\(F\subseteq B\cap\waybelowset{x}\). For each \(b\in F\), interpolation
yields an element \(c_b\in X\) such that
\(b\below c_b\below x\). Since \(x=\bigvee_{i\in I}x_i\), there exists
\(i_b\in I\) with \(c_b\le x_{i_b}\).

Because \(F\) is finite and \((x_i)_{i\in I}\) is directed, there is an
index \(i_0\in I\) such that \(x_{i_b}\le x_{i_0}\) for every \(b\in F\).
Thus \(b\below c_b\le x_{i_0}\), and hence
\(b\below x_{i_0}\) for every \(b\in F\). Therefore
\(F\subseteq B\cap\waybelowset{x_{i_0}}\), so
\(F\in s(x_{i_0})\).

Consequently,
\[
s(x)=\bigcup_{i\in I}s(x_i).
\]
By Proposition~\ref{prop:ideal-completion}, directed suprema in \(L\)
are given by unions. Hence
\[
s\left(\bigvee_{i\in I}x_i\right)
=
\bigvee_{i\in I}s(x_i),
\]
and therefore \(s\) is Scott-continuous.
\end{proof}


\begin{lemma}
\label{lemma:r-well-defined}
For every \(I\in L\), the set \(\bigcup I\) is directed in \(X\).
Consequently, the assignment in
Construction~\ref{construction:section-retraction} defines a map
\(r:L\to X\).
\end{lemma}

\begin{proof}
Since \(I\) is a nonempty ideal of \(P\), choose \(F\in I\). Every
member of \(P\) contains \(\bot\), so \(\bot\in F\subseteq\bigcup I\).
Thus \(\bigcup I\) is nonempty.

Now let \(b_1,b_2\in\bigcup I\). There exist \(F_1,F_2\in I\) such that
\(b_1\in F_1\) and \(b_2\in F_2\). Since \(I\) is directed in the
inclusion order, there exists \(F_3\in I\) with
\(F_1,F_2\subseteq F_3\). As \(F_3\in P\), it is directed in \(X\);
hence there exists \(b_3\in F_3\) such that
\(b_1,b_2\le b_3\). Since \(F_3\subseteq\bigcup I\), we have
\(b_3\in\bigcup I\).

Therefore \(\bigcup I\) is directed in \(X\). Since \(X\) is a dcpo,
the supremum \(\bigvee_X\bigcup I\) exists, and \(r\) is well defined.
\end{proof}


\begin{lemma}
\label{lemma:r-scott-continuous}
The map \(r:L\to X\) is Scott-continuous.
\end{lemma}

\begin{proof}
First, \(r\) is monotone. Indeed, if \(I\subseteq J\), then
\(\bigcup I\subseteq\bigcup J\), and hence
\(r(I)\le r(J)\).

Now let \((I_j)_{j\in J}\) be a directed family in \(L\). By
Proposition~\ref{prop:ideal-completion}, its supremum is
\(\bigcup_{j\in J} I_j\). Therefore
\[
r\left(\bigvee_{j\in J} I_j\right)
=
\bigvee_X \bigcup\left(\bigcup_{j\in J} I_j\right)
=
\bigvee_X \bigcup_{j\in J}\bigcup I_j.
\]

Since \(r\) is monotone, the family \((r(I_j))_{j\in J}\) is directed.
Moreover, \(\bigvee_{j\in J}r(I_j)\) is the supremum in \(X\) of the
same set \(\bigcup_{j\in J}\bigcup I_j\). Indeed, it is an upper bound
of this set because each \(\bigcup I_j\) lies below \(r(I_j)\);
conversely, every upper bound of
\(\bigcup_{j\in J}\bigcup I_j\) is an upper bound of each
\(r(I_j)=\bigvee_X\bigcup I_j\). Hence
\[
r\left(\bigvee_{j\in J} I_j\right)
=
\bigvee_{j\in J} r(I_j).
\]
Thus \(r\) preserves directed suprema and is Scott-continuous.
\end{proof}


\begin{proposition}
\label{prop:section-retraction}
The Scott-continuous maps \(s:X\to L\) and \(r:L\to X\) satisfy
\(r\circ s=\id_X\). Moreover,
\[
s(\topel)=\topel_L.
\]
Thus \(r\) is a retraction with section \(s\).
\end{proposition}

\begin{proof}
We first show that
\[
\bigcup s(x)
=
B\cap\waybelowset{x}
\qquad(x\in X).
\]
The inclusion from left to right follows immediately from the definition
of \(s(x)\).

Conversely, let \(b\in B\cap\waybelowset{x}\). If \(b=\bot\), then
\(\{\bot\}\in s(x)\). If \(b\ne\bot\), then
\(\{\bot,b\}\) is a finite directed subset of \(B\), contains \(\bot\),
and is contained in \(B\cap\waybelowset{x}\). Hence
\(\{\bot,b\}\in s(x)\). In either case, \(b\in\bigcup s(x)\), proving
the reverse inclusion.

Since \(B\) is a basis of \(X\), it follows that
\[
r(s(x))
=
\bigvee_X\bigcup s(x)
=
\bigvee_X\bigl(B\cap\waybelowset{x}\bigr)
=
x.
\]
Therefore \(r\circ s=\id_X\).

Finally, Corollary~\ref{cor:everything-below-top} gives
\(b\below\topel\) for every \(b\in B\). Hence
\(B\cap\waybelowset{\topel}=B\). Since every \(F\in P\) is a subset of
\(B\), the definition of \(s\) yields
\[
s(\topel)
=
P
=
\topel_L.
\]
\end{proof}


The source and target domains are now linked by a Scott-continuous
section--retraction pair satisfying \(s(\topel)=\topel_L\). We next
analyse the elements \(u\below s\) in the function space \([X\to L]\).
The local finiteness of \(L\), together with evaluation at the greatest
element \(\topel\), will imply that every such approximant \(u\) has
finite image.

\subsection{Finite-Image Approximations of the Section}
\label{subsec:finite-image-approximations}

Having established a Scott-continuous section--retraction pair between
\(X\) and \(L\), we now investigate the approximants of the section map
\(s\) in the Scott function space \([X\to L]\). The key result of this
subsection is that every approximant of \(s\) has finite image. This
will later allow the retraction \(r\) to transform these approximants
into finite-image deflations on \(X\), leading to the desired
contradiction.


\begin{lemma}
\label{lemma:top-approximation}
If \(u\below s\) in \([X\to L]\), then
\(u(\topel)\below\topel_L\).
\end{lemma}

\begin{proof}
Let \(A\subseteq L\) be directed and suppose that
\(\topel_L\le\bigvee A\). Since \(\topel_L\) is the greatest element of
\(L\), it follows that \(\bigvee A=\topel_L\).

For each \(a\in A\), let \(\bar a:X\to L\) denote the constant map with
value \(a\). The family \(\{\bar a:a\in A\}\) is directed in
\([X\to L]\), and its pointwise supremum is the constant map with value
\(\topel_L\). Since \(s\le\bar{\topel_L}\), we have
\[
s
\le
\bigvee_{a\in A}\bar a.
\]

As \(u\below s\), there exists \(a\in A\) such that \(u\le\bar a\).
Evaluating at \(\topel\) gives \(u(\topel)\le a\). This is precisely the
defining condition for \(u(\topel)\below\topel_L\).
\end{proof}


\begin{lemma}
\label{lemma:finite-image}
If \(u\below s\) in \([X\to L]\), then \(u[X]\) is finite.
\end{lemma}

\begin{proof}
Since \(\topel\) is the greatest element of \(X\) and \(u\) is
monotone, \(u(x)\le u(\topel)\) for every \(x\in X\). Hence
\[
u[X]\subseteq\downarrow_Lu(\topel).
\]

By Lemma~\ref{lemma:top-approximation},
\(u(\topel)\below\topel_L\). Lemma~\ref{lemma:finite-principal-lowersets}
therefore implies that \(\downarrow_Lu(\topel)\) is finite. Consequently,
\(u[X]\) is finite.
\end{proof}


The preceding lemmas show that every approximant \(u\below s\) has
finite image. In the next subsection, we combine this fact with the
Scott-continuous retraction \(r:L\to X\). Assuming that the function
space \([X\to L]\) is continuous, the section \(s\) can be recovered as
the directed supremum of its approximants; composing these approximants
with \(r\) will then yield a directed family of finite-image deflations
whose supremum is \(\id_X\). This would make \(X\) an RB-domain,
contradicting its construction.

\subsection{The Counterexample}
\label{subsec:counterexample-proof}

Having shown that every approximant of the section map \(s\) has finite
image, we now derive the desired contradiction. If the Scott function
space \([X\to L]\) were continuous, then \(s\) would be the directed
supremum of the elements way below it. Composing these approximants with
the Scott-continuous retraction \(r:L\to X\) will produce a directed
family of finite-image deflations on \(X\) whose supremum is
\(\id_X\). This would force \(X\) to be an RB-domain, contrary to its
construction.

\begin{theorem}
\label{thm:continuity-implies-RB}
If the Scott function space \([X\to L]\) is continuous, then \(X\) is
an RB-domain.
\end{theorem}

\begin{proof}
Assume that the Scott function space \([X\to L]\) is continuous. Since
\(s\in[X\to L]\), we have
\[
s
=
\bigvee
\{\,u\in[X\to L]:u\below s\,\},
\]
where the supremum is directed.

For each \(u\below s\), define
\[
f_u
:=
r\circ u
:
X\to X.
\]
Since both \(r\) and \(u\) are Scott-continuous,
\(f_u\) is Scott-continuous. Moreover, composition preserves the pointwise
order, so the family
\(\{f_u:u\below s\}\)
is directed.

For every \(u\below s\), we have \(u\le s\). Since \(r\) is monotone
and \(r\circ s=\id_X\), it follows that
\[
f_u
=
r\circ u
\le
r\circ s
=
\id_X.
\]
Thus each \(f_u\) is a deflation. Moreover, Lemma~\ref{lemma:finite-image}
shows that \(u[X]\) is finite, and hence
\[
f_u[X]
=
r[u[X]]
\]
is finite.

It remains to compute the supremum of this directed family. For every
\(x\in X\), Scott continuity of \(r\) gives
\[
\begin{aligned}
\left(\bigvee_{u\below s}f_u\right)(x)
&=
\bigvee_{u\below s}r(u(x)) \\
&=
r\left(\bigvee_{u\below s}u(x)\right) \\
&=
r(s(x)) \\
&=
x.
\end{aligned}
\]
Therefore
\[
\bigvee_{u\below s}f_u
=
\id_X.
\]

We have thus obtained a directed family of Scott-continuous
finite-image deflations on \(X\) whose supremum is \(\id_X\).
Consequently, \(X\) is an RB-domain.
\end{proof}

\begin{theorem}[Main Counterexample]
\label{thm:main-counterexample}
There exist an FS-domain \(X\) and an algebraic domain \(L\) such that
the Scott function space \([X\to L]\) is not continuous.
\end{theorem}

\begin{proof}
By Lemma~\ref{lemma:compact-top-fs}, the domain \(X=D^{\top}\) is an
FS-domain. Since the original domain \(D\) is not an RB-domain,
Lemma~\ref{lemma:compact-top-nonrb} shows that \(X\) is not an
RB-domain. By Proposition~\ref{prop:algebraic-target}, the domain
\(L\) constructed in Section~\ref{subsec:algebraic-target} is
algebraic.

Suppose, for contradiction, that the Scott function space
\([X\to L]\) were continuous. Then
Theorem~\ref{thm:continuity-implies-RB} would imply that \(X\) is an
RB-domain, contradicting the preceding paragraph. Therefore,
\([X\to L]\) is not continuous.
\end{proof}

\begin{corollary}
\label{cor:AJ-counterexample}
Proposition~4.2.10 of
\cite{AbramskyJung1994}
is false.
\end{corollary}

\begin{proof}
Proposition~4.2.10 asserts that the Scott function space
\([D\to E]\) is continuous whenever \(D\) is an FS-domain and \(E\) is
pointed and continuous. Since every algebraic domain is pointed and
continuous, Theorem~\ref{thm:main-counterexample} provides an
FS-domain \(X\) and an algebraic domain \(L\) such that
\([X\to L]\) is not continuous. Hence Proposition~4.2.10 is false.
\end{proof}

\section{FS Sources and Plotkin's Tie as Target}
\label{sec:tie-target}

The counterexample constructed in the previous section shows that
continuity of Scott function spaces cannot be guaranteed under
one-sided FS assumptions in general. It is therefore natural to ask
whether the failure arises from the particular choice of target domain,
or whether it is an unavoidable phenomenon. In this section, we show
that the negative result is far from universal. When the target is
Plotkin's tie, every Scott function space \([D\to T]\) is continuous
whenever \(D\) is an FS-domain, despite the fact that Plotkin's tie is
itself neither an FS-domain nor an RB-domain.

The proof reveals a remarkable local approximation property of
Plotkin's tie. Although \(T\) is not globally approximated by finite
separated deflations, every compact approximation to a Scott-continuous
map into \(T\) is confined to a suitable finite truncation of the tie.
This localization allows the continuity argument to be carried out
within finite layers, whose directed union then recovers the original
function.

Our proof proceeds in five steps. We first introduce the finite-layer
truncations of Plotkin's tie. We then establish a local truncation
lemma showing that every compact approximation lands inside one of
these finite layers. Next, we prove that these approximants form a
directed family, allowing each Scott-continuous map to be recovered as
their supremum. This yields the continuity of the Scott function space
\([D\to T]\). Finally, we show that, despite this positive result,
Plotkin's tie remains neither an FS-domain nor an RB-domain.

\subsection{Finite-layer truncation maps}
\label{subsec:finite-layer-truncation-maps}

The key idea underlying the positive result of this section is that,
although Plotkin's tie is infinite, every compact approximation to a
Scott-continuous map into \(T\) interacts with only finitely many of its
levels. Rather than working directly with the entire domain, we
therefore introduce a family of finite-layer truncation maps that
preserve the initial levels of the tie while collapsing all higher
levels to a single value. These truncations retain enough of the
approximation structure to capture every compact approximation, yet have
finite image.

The finite-layer truncation maps will play the same role in this
section as the finite-image approximants of the section map did in the
counterexample of Section~\ref{sec:counterexample}. They provide the
finite approximations from which the continuity of the Scott function
space will ultimately be recovered.

We begin by defining the truncation maps.
For each integer \(n\ge0\), let \(T_n\) denote the finite subposet of
Plotkin's tie consisting of all elements lying in the first \(n\)
branching levels. Thus
\[
T_0\subseteq T_1\subseteq T_2\subseteq\cdots\subseteq T,
\]
and
\[
T=\bigcup_{n=0}^{\infty}T_n.
\]

Each \(T_n\) is finite and therefore forms a finite algebraic dcpo with
the induced order.

For every \(n\ge0\), we define the corresponding finite-layer
truncation map
\[
\tau_n:T\to T,
\]
which fixes every element of \(T_n\) and collapses every element lying
strictly above the \(n\)-th layer to the unique maximal element of
\(T_n\) lying below it.

\begin{definition}
\label{def:finite-layer-truncation}
For each \(N\in\mathbb{N}\), define the finite-layer truncation map
\(\tau_N:T\to T\) by
\[
\tau_N(\bottom_T)=\bottom_T,
\]
\[
\tau_N(a_n)=a_n
\quad\text{and}\quad
\tau_N(b_n)=b_n
\qquad (n\le N),
\]
and
\[
\tau_N(a_n)=\tau_N(b_n)=a_{N+1}
\qquad (n>N),
\]
together with
\[
\tau_N(\topel_T)=a_{N+1}.
\]
Thus \(\tau_N\) fixes the first \(N+1\) levels of \(T\) and collapses
all higher levels, as well as the greatest element, to \(a_{N+1}\).
\end{definition}

\begin{lemma}
\label{lemma:truncation-continuous}
For every \(N\in\mathbb{N}\), the map \(\tau_N\) is
Scott-continuous, preserves the least element, and has finite image.
\end{lemma}

\begin{proof}
By definition, \(\tau_N(\bottom_T)=\bottom_T\). We first show that
\(\tau_N\) is monotone. The only case requiring attention is when
\(s\le t\), where \(s\) lies at a level at most \(N\) and \(t\) lies
above level \(N\). In this case,
\[
\tau_N(s)=s\le a_{N+1}=\tau_N(t).
\]
All other cases follow immediately from the definition.

Let \(A\subseteq T\) be directed. Suppose first that
\(\bigvee A\ne\topel_T\). Since every element of \(T\) other than
\(\topel_T\) is compact, there exists \(a\in A\) such that
\(a=\bigvee A\). By monotonicity,
\[
\tau_N\left(\bigvee A\right)
=
\tau_N(a)
=
\bigvee\tau_N[A].
\]

Suppose instead that \(\bigvee A=\topel_T\). Then either
\(\topel_T\in A\), or the levels of the elements of \(A\) are
unbounded. In either case, \(a_{N+1}\in\tau_N[A]\). Moreover,
\(a_{N+1}\) is the greatest element of \(\tau_N[A]\), and hence
\[
\tau_N\left(\bigvee A\right)
=
a_{N+1}
=
\bigvee\tau_N[A].
\]
Therefore \(\tau_N\) preserves directed suprema and is
Scott-continuous.

Finally,
\[
\tau_N[T]
\subseteq
\{\bottom_T\}
\cup
\{a_n,b_n:0\le n\le N\}
\cup
\{a_{N+1}\},
\]
so \(\tau_N[T]\) is finite.
\end{proof}

\begin{remark}
\label{remark:truncation-not-deflation}
The truncation map \(\tau_N\) is not a deflation, since
\[
\tau_N(b_{N+1})
=
a_{N+1}
\not\le
b_{N+1}.
\]
Thus finite image alone does not imply that \(\tau_N\) is way below the
identity. In the argument below, the required way-below relation will
instead arise from finite separation in the source FS-domain.
\end{remark}

The finite-layer truncation maps provide finite-image approximations to
Plotkin's tie, but they are not themselves deflations. Their role is
instead to localize the approximation problem: after composing with a
suitable finite approximation arising from the FS structure of the
source, a map into \(T\) can be forced to remain within one finite
layer. We next make this localization precise in the local truncation
lemma.

\subsection{The local truncation lemma}
\label{subsec:local-truncation}

The finite-layer truncation maps introduced in the previous subsection
cannot approximate the identity on \(T\) uniformly, since no fixed
finite layer contains all of Plotkin's tie. The crucial observation,
however, is that such a global approximation is unnecessary. For the
purpose of proving continuity of the Scott function space, it suffices
to approximate each compact map individually.

The key idea is that every compact approximation to a
Scott-continuous map into \(T\) has image contained in some finite
layer of the tie. Although the required layer depends on the
particular approximation, it is always finite. This locality property
allows the infinite approximation structure of \(T\) to be reduced,
one compact approximation at a time, to the finite-layer truncation
maps constructed above.

The following lemma makes this principle precise.

\begin{lemma}[Local truncation]
\label{lemma:local-truncation}
Let \(D\) be an FS-domain, and let
\((f_i,M_i)_{i\in I}\) be a finitely separating approximate identity
on \(D\). Let \(h:D\to T\) be Scott-continuous. Fix \(i\in I\), and
choose \(N\in\mathbb{N}\) such that
\[
m\in M_i
\quad\text{and}\quad
h(m)\ne\topel_T
\quad\Longrightarrow\quad
\rank(h(m))\le N.
\]
Define
\[
p_{i,N}
:=
\tau_N\circ h\circ f_i.
\]
Then \(p_{i,N}\below h\) in \([D\to T]\). In particular,
\(p_{i,N}\le h\), and \(p_{i,N}\) has finite image.
\end{lemma}

\begin{proof}
The map \(p_{i,N}\) is Scott-continuous and has finite image by
Lemma~\ref{lemma:truncation-continuous}.

We first show that \(p_{i,N}\le h\). Let \(x\in D\). Since \(M_i\)
finitely separates \(f_i\) from \(\id_D\), there exists \(m\in M_i\)
such that
\[
f_i(x)\le m\le x.
\]
Suppose first that \(h(m)\ne\topel_T\). By the choice of \(N\), the
element \(h(m)\) lies at a level at most \(N\). Since
\(h(f_i(x))\le h(m)\), the element \(h(f_i(x))\) is also fixed by
\(\tau_N\). Hence
\[
p_{i,N}(x)
=
h(f_i(x))
\le
h(x).
\]
If \(h(m)=\topel_T\), then \(m\le x\) gives
\(\topel_T=h(m)\le h(x)\), and therefore \(h(x)=\topel_T\). Thus again
\(p_{i,N}(x)\le h(x)\).

It remains to prove that \(p_{i,N}\below h\). Let
\((g_\lambda)_{\lambda\in\Lambda}\) be directed in \([D\to T]\), and
suppose that
\[
h
\le
\bigvee_{\lambda\in\Lambda}g_\lambda.
\]
For every \(m\in M_i\), we have
\[
\tau_N(h(m))\below h(m).
\]
Indeed, if \(h(m)\ne\topel_T\), then the choice of \(N\) gives
\(\tau_N(h(m))=h(m)\), and this element is compact. If
\(h(m)=\topel_T\), then
\(\tau_N(h(m))=a_{N+1}\below\topel_T\).

Consequently, for each \(m\in M_i\), there exists
\(\lambda_m\in\Lambda\) such that
\[
\tau_N(h(m))
\le
g_{\lambda_m}(m).
\]
Since \(M_i\) is finite and the family \((g_\lambda)\) is directed,
there exists a single \(\lambda\in\Lambda\) such that
\[
\tau_N(h(m))
\le
g_\lambda(m)
\qquad
(m\in M_i).
\]

Now let \(x\in D\), and choose \(m\in M_i\) with
\(f_i(x)\le m\le x\). Then
\[
\begin{aligned}
p_{i,N}(x)
&=
\tau_N(h(f_i(x)))\\
&\le
\tau_N(h(m))\\
&\le
g_\lambda(m)\\
&\le
g_\lambda(x).
\end{aligned}
\]
Thus \(p_{i,N}\le g_\lambda\), proving that
\(p_{i,N}\below h\).
\end{proof}

The local truncation lemma provides a systematic supply of finite-image
maps way below a given \(h\in[D\to T]\). The truncation level may depend
on both the chosen source approximation \(f_i\) and the finitely many
values of \(h\) on its separating set \(M_i\), but no global bound is
required.

To prove continuity, however, we must understand not only these specially
constructed approximants, but the entire set of maps way below \(h\).
We next show that every such map is bounded by a constant map at some
finite level of the tie. This boundedness will allow any two way-below
approximants to be dominated by a common local truncation.


\subsection{Directedness of the way-below approximants}
\label{subsec:directedness-approximants}

The local truncation lemma constructs finite-image approximants below
\(h\), but continuity requires the full set
\(\waybelowset h\) to be directed. The first step is to show that every
map way below \(h\) is uniformly bounded at some finite level of
Plotkin's tie.

\begin{lemma}
\label{lemma:waybelow-finite-level}
Let \(h\in[D\to T]\). If \(u\below h\), then there exists
\(r\in\mathbb{N}\) such that
\[
u\le \bar a_r.
\]
Consequently, \(u[D]\) is finite and does not contain \(\topel_T\).
\end{lemma}

\begin{proof}
The constant maps form an increasing chain
\[
\bar a_0\le \bar a_1\le\cdots
\]
whose supremum is the constant map \(\bar{\topel_T}\). Since
\(h\le\bar{\topel_T}\), we have
\[
h
\le
\bigvee_{r\in\mathbb{N}}\bar a_r.
\]
As \(u\below h\), there exists \(r\in\mathbb{N}\) such that
\(u\le\bar a_r\). Hence
\[
u[D]\subseteq\downarrow_Ta_r.
\]
The principal lower set \(\downarrow_Ta_r\) is finite and does not
contain \(\topel_T\). Therefore \(u[D]\) is finite and
\(\topel_T\notin u[D]\).
\end{proof}
\begin{lemma}
\label{lemma:directed-waybelow}
For every \(h\in[D\to T]\), the set
\(\waybelowset h\) is directed.
\end{lemma}

\begin{proof}
Let \(u,v\below h\). Since
\[
h
=
\bigvee_{i\in I} h\circ f_i
\]
and the family \((h\circ f_i)_{i\in I}\) is directed, there exist
\(i_u,i_v\in I\) such that
\[
u\le h\circ f_{i_u}
\qquad\text{and}\qquad
v\le h\circ f_{i_v}.
\]
By directedness of the family \((f_i)_{i\in I}\), we may choose
\(i\in I\) such that
\[
f_{i_u},f_{i_v}\le f_i.
\]
Since \(h\) is monotone, it follows that
\[
u,v\le h\circ f_i.
\]

By Lemma~\ref{lemma:waybelow-finite-level}, the images \(u[D]\) and
\(v[D]\) are finite and do not contain \(\topel_T\). Choose
\(N\in\mathbb{N}\) sufficiently large that

\begin{enumerate}
\item every nonbottom element of \(u[D]\cup v[D]\) has rank at most
\(N\); and
\item every element \(m\in M_i\) for which \(h(m)\ne\topel_T\) satisfies
\(\rank(h(m))\le N\).
\end{enumerate}

Define
\[
p
:=
\tau_N\circ h\circ f_i.
\]
By Lemma~\ref{lemma:local-truncation}, we have \(p\below h\).

We claim that \(u\le p\). Let \(x\in D\), and put
\(t=h(f_i(x))\). Since \(u\le h\circ f_i\), we have \(u(x)\le t\).

If \(t\) lies at a level at most \(N\), then \(\tau_N(t)=t\), and hence
\[
u(x)\le t=p(x).
\]

If \(t\) lies above level \(N\), or if \(t=\topel_T\), then \(u(x)\) is
either \(\bottom_T\) or lies at a level at most \(N\). Therefore
\[
u(x)\le a_{N+1}=\tau_N(t)=p(x).
\]
Thus \(u\le p\). The same argument gives \(v\le p\).

Hence every two elements of \(\waybelowset h\) have an upper bound in
\(\waybelowset h\). Moreover, \(\waybelowset h\) is nonempty by
Lemma~\ref{lemma:local-truncation}. Therefore
\(\waybelowset h\) is directed.
\end{proof}

The previous two lemmas establish the essential approximation
properties of Scott-continuous maps into Plotkin's tie. Every compact
approximant is localized within a finite layer of the tie, and these
approximants form a directed family. We are therefore ready to prove
that every Scott-continuous map is the directed supremum of its
way-below approximants, thereby establishing the continuity of the
Scott function space \([D\to T]\).

\subsection{Continuity of the function space}
\label{subsec:continuity-function-space}

The preceding subsections establish the two ingredients required for
continuity. The local truncation lemma provides a directed family of
finite-image approximants below every Scott-continuous map, while the
directedness lemma shows that these approximants form a directed set.
The remainder of the argument follows the standard characterization of
continuous dcpos: we prove that every Scott-continuous map is the
directed supremum of the maps way below it. This establishes the
continuity of the Scott function space \([D\to T]\).

\begin{theorem}
\label{thm:tie-target-continuous}
If \(D\) is an FS-domain, then the Scott function space
\([D\to T]\) is continuous.
\end{theorem}

\begin{proof}
Let \(h\in[D\to T]\). By
Lemma~\ref{lemma:directed-waybelow}, the set
\(\waybelowset h\) is directed. Since every \(u\below h\) satisfies
\(u\le h\), we have
\[
\bigvee\waybelowset h\le h.
\]
It remains to prove the reverse inequality.

Fix \(x\in D\), and let \(t\below h(x)\). Since
\((f_i)_{i\in I}\) is an approximate identity on \(D\),
\[
x=\bigvee_{i\in I}f_i(x).
\]
Scott continuity of \(h\) therefore gives
\[
h(x)=\bigvee_{i\in I}h(f_i(x)).
\]
As \(t\below h(x)\), there exists \(i\in I\) such that
\[
t\le h(f_i(x)).
\]

Choose \(N\in\mathbb{N}\) sufficiently large that

\begin{enumerate}
\item every \(m\in M_i\) with \(h(m)\ne\topel_T\) satisfies
\(\rank(h(m))\le N\); and
\item \(t=\bottom_T\), or \(\rank(t)\le N\).
\end{enumerate}

Set
\[
p_{i,N}:=\tau_N\circ h\circ f_i.
\]
By Lemma~\ref{lemma:local-truncation},
\(p_{i,N}\below h\).

We claim that \(t\le p_{i,N}(x)\). Put \(y=h(f_i(x))\), so that
\(t\le y\). If \(y\) lies at a level at most \(N\), then
\(\tau_N(y)=y\), and hence
\[
t\le y=p_{i,N}(x).
\]
If \(y\) lies above level \(N\), or \(y=\topel_T\), then
\(\tau_N(y)=a_{N+1}\). Since \(t=\bottom_T\) or \(t\) lies at a level
at most \(N\), we again obtain
\[
t\le a_{N+1}=p_{i,N}(x).
\]

Thus every \(t\below h(x)\) lies below \(u(x)\) for some
\(u\below h\). Since \(T\) is continuous,
\[
h(x)
=
\bigvee\waybelowset{h(x)}
\le
\bigvee_{u\below h}u(x).
\]
The reverse inequality follows from \(u\le h\) for every
\(u\below h\). Therefore
\[
h(x)=\bigvee_{u\below h}u(x).
\]
Since this holds for every \(x\in D\), directed suprema in
\([D\to T]\) being computed pointwise give
\[
h=\bigvee\waybelowset h.
\]

Hence every element of \([D\to T]\) is the directed supremum of the
elements way below it. Therefore \([D\to T]\) is continuous.
\end{proof}

Theorem~\ref{thm:tie-target-continuous} shows that Plotkin's tie behaves
exceptionally well as a target: continuity of \([D\to T]\) holds for
every FS-domain \(D\), even though the corresponding one-sided closure
statement fails for general pointed continuous targets. This positive
result might suggest that \(T\) itself belongs to one of the standard
classes supporting function-space closure. We conclude the section by
showing that this is not the case: Plotkin's tie is neither an
FS-domain nor an RB-domain.

\subsection{Failure of the FS and RB properties}
\label{subsec:tie-not-fs}

Theorem~\ref{thm:tie-target-continuous} establishes that Plotkin's tie
behaves remarkably well as the target of Scott function spaces.
One might therefore expect this behaviour to stem from membership in one
of the familiar cartesian-closed classes of domains. Surprisingly, this
is not the case. In this subsection, we show that Plotkin's tie is
neither an FS-domain nor an RB-domain. Consequently, the continuity of
\([D\to T]\) for every FS-domain \(D\) cannot be explained by the
existing closure theory, but instead reflects the distinctive local
approximation structure of Plotkin's tie.

\begin{theorem}
\label{thm:tie-not-fs-rb}
Plotkin's tie \(T\) is neither an FS-domain nor an RB-domain.
\end{theorem}

\begin{proof}
Recall that \(T\) is an algebraic domain but is not bifinite. On the
other hand, every algebraic FS-domain is bifinite
\cite{AbramskyJung1994}. It follows that \(T\) cannot be an FS-domain.

Moreover, every RB-domain is an FS-domain. Indeed, an RB-domain is a
Scott-continuous retract of a bifinite domain, every bifinite domain is
an FS-domain, and the class of FS-domains is closed under
Scott-continuous retracts
\cite[Proposition~4.2.12]{AbramskyJung1994}. Since \(T\) is not an
FS-domain, it cannot be an RB-domain.
\end{proof}

Theorem~\ref{thm:tie-not-fs-rb}, together with
Theorem~\ref{thm:tie-target-continuous}, shows that the continuity of
\([D\to T]\) for every FS-domain \(D\) is not a consequence of
Plotkin's tie belonging to one of the standard cartesian-closed classes.
Rather, it arises from the local truncation property established above:
each way-below approximation to a map into \(T\) can be confined to a
suitable finite layer, even though no global FS- or RB-approximation of
\(T\) exists.


\section{Plotkin's Tie as Source and FS Targets}
\label{sec:tie-source}

The preceding section showed that Plotkin's tie behaves exceptionally
well as the target of Scott function spaces with FS sources. We now
turn to the reverse configuration and consider Scott-continuous maps
from \(T\) into an arbitrary FS-domain \(L\). Although the conclusion is
again that the function space is continuous, the mechanism is entirely
different. The finite-layer truncations used when \(T\) was the target
are no longer available on the codomain side; instead, the proof must
exploit the finite separation structure of \(L\) together with the
particular order-theoretic shape of the source \(T\).

The first key observation is that a Scott-continuous map
\(r:L\to L\) that is finitely separated from \(\id_L\) satisfies
\(r(x)\below x\) for every \(x\in L\). This pointwise approximation
property does not by itself imply that \(r\circ h\below h\) for an
arbitrary source domain. For Plotkin's tie, however, the infinitely
many pointwise conditions can be reduced to finitely many initial
values together with the two tails represented by the ascending chains
\((a_n)_{n\in\mathbb N}\) and \((b_n)_{n\in\mathbb N}\), both of which
have supremum \(\topel_T\). This finite reduction yields the uniform
way-below estimate required in the function space.

Our proof proceeds in four stages. We first show that finite separation
in the target induces pointwise way-below approximation. We then use
the two-branch structure of Plotkin's tie to strengthen this pointwise
information to a function-space estimate. Next, we establish the
directedness of the way-below approximants. Finally, applying these
results to a finitely separating approximate identity on \(L\), we
prove that every element of \([T\to L]\) is the directed supremum of
the maps way below it, and hence that \([T\to L]\) is continuous.

\subsection{Finite separation induces pointwise approximation}
\label{subsec:pointwise-approximation}

The proof of continuity for the Scott function space \([T\to L]\)
begins with a simple but fundamental observation about finitely
separated maps. Let \(L\) be an FS-domain, and let
\(r:L\to L\) be Scott-continuous. If \(r\) is finitely separated from
\(\id_L\), then every point of \(L\) is approximated by its image under
\(r\) in the way-below relation. Thus finite separation, which is
defined globally in terms of a finite separating set, immediately
produces a pointwise approximation property.

This pointwise estimate is the starting point of the argument.
Although it is insufficient to conclude that
\(r\circ h\below h\) for an arbitrary source domain, it provides the
local approximation from which the special order-theoretic structure of
Plotkin's tie will later allow a global function-space approximation to
be recovered.

We begin by making this pointwise approximation property precise.
\begin{lemma}
\label{lem:cofinal-colour}
Let \(A\) be a directed set, let \(M\) be finite, and let
\(c:A\to M\).
Then there exists \(m\in M\) such that
\[
\{a\in A:c(a)=m\}
\]
is cofinal in \(A\).
\end{lemma}

\begin{proof}
For each \(m\in M\), set
\[
A_m:=\{a\in A:c(a)=m\}.
\]
Suppose, for contradiction, that no \(A_m\) is cofinal in \(A\).
Then, for every \(m\in M\), there exists \(a_m\in A\) such that no
element of \(A_m\) lies above \(a_m\).

Since \(M\) is finite and \(A\) is directed, there exists \(a\in A\)
such that \(a_m\le a\) for every \(m\in M\). Let \(m_0:=c(a)\). Then
\(a\in A_{m_0}\) and \(a_{m_0}\le a\), contradicting the choice of
\(a_{m_0}\).

Hence \(A_m\) is cofinal in \(A\) for some \(m\in M\).
\end{proof}

\begin{lemma}
\label{lemma:pointwise-waybelow}
Let \(L\) be a dcpo, and let \(r:L\to L\) be a Scott-continuous map
that is finitely separated from \(\id_L\). Then
\(
r(x)\below x
\)
for every \(x\in L\).
\end{lemma}

\begin{proof}
Let \(x\in L\), and let \(A\subseteq L\) be directed with
\(x\le\bigvee A\). We must show that \(r(x)\le a\) for some
\(a\in A\).

Let \(M\) be a finite separating set for \(r\). For each \(a\in A\),
choose an element \(m(a)\in M\) satisfying
\begin{equation}
\label{eq:separator-at-a}
r(a)\le m(a)\le a.
\end{equation}

By Lemma~\ref{lem:cofinal-colour}, there exists
\(m\in M\) such that the subset
\[
A_m
=
\{\,a\in A:m(a)=m\,\}
\]
is cofinal in \(A\).

Since \(r\) is Scott-continuous and \(A_m\) is cofinal in \(A\),
\[
\begin{aligned}
r(x)
&\le
r\!\left(\bigvee A\right)\\
&=
\bigvee_{a\in A}r(a)\\
&=
\bigvee_{a\in A_m}r(a)\\
&\le
m.
\end{aligned}
\]

Now choose any \(a\in A_m\). By
\eqref{eq:separator-at-a},
\(
m\le a.
\)
Hence
\[
r(x)\le m\le a.
\]
Therefore \(r(x)\below x\).
\end{proof}

\subsection{Controlling maps on the two branches}
\label{subsec:two-branches}

The pointwise approximation property established in the previous
subsection does not immediately imply that
\(r\circ h\below h\) in the Scott function space \([T\to L]\).
Indeed, for a general source domain, verifying the way-below relation
would require controlling the values of \(h\) on infinitely many
independent points.

The essential feature of Plotkin's tie is that its infinite structure
is concentrated in two ascending chains,
\[
a_0\le a_1\le\cdots\le\topel_T,
\qquad
b_0\le b_1\le\cdots\le\topel_T,
\]
whose suprema are both equal to the greatest element
\(\topel_T\). Consequently, Scott continuity implies that the behaviour
of \(h\) on sufficiently high levels of each branch is already
determined by its value at \(\topel_T\). Only finitely many initial
elements of the two branches therefore require separate attention.

This observation allows the infinitely many pointwise estimates
\(r(h(x))\below h(x)\) to be reduced to finitely many conditions,
leading to a global way-below approximation in the function space.
The following lemma makes this reduction precise.

\begin{lemma}
\label{lemma:two-branch-control}
Let \(L\) be a continuous dcpo, let \(r:L\to L\) be a
Scott-continuous map that is finitely separated from \(\id_L\), and
let \(h:T\to L\) be Scott-continuous. Suppose that
\(\mathcal D\subseteq[T\to L]\) is directed and
\[
h\le\bigvee\mathcal D.
\]
Then there exist \(g_0\in\mathcal D\) and \(N\in\mathbb N\) such that
\[
r(h(\topel_T))\le g_0(\topel_T),
\]
and, for every \(n\ge N\),
\[
r(h(a_n))\le g_0(a_n)
\qquad\text{and}\qquad
r(h(b_n))\le g_0(b_n).
\]
\end{lemma}

\begin{proof}
Set
\(
z:=h(\topel_T).
\)
By Lemma~\ref{lemma:pointwise-waybelow},
\(r(z)\below z\). Since \(L\) is continuous, interpolation gives
\(u\in L\) such that
\[
r(z)\below u\below z.
\]

Evaluating \(h\le\bigvee\mathcal D\) at \(\topel_T\), we obtain
\[
z
\le
\bigvee_{g\in\mathcal D}g(\topel_T).
\]
Since \(u\below z\), there exists \(g_0\in\mathcal D\) such that
\(
u\le g_0(\topel_T).
\)

The two branches of Plotkin's tie have supremum \(\topel_T\). Hence
Scott continuity of \(g_0\) gives
\[
g_0(\topel_T)
=
\bigvee_{n\in\mathbb N}g_0(a_n)
=
\bigvee_{n\in\mathbb N}g_0(b_n).
\]
Because
\[
r(z)\below u\le g_0(\topel_T),
\]
there exist \(N_a,N_b\in\mathbb N\) such that
\[
r(z)\le g_0(a_{N_a})
\qquad\text{and}\qquad
r(z)\le g_0(b_{N_b}).
\]
Let
\(
N:=\max\{N_a,N_b\}.
\)
By monotonicity of \(g_0\), for every \(n\ge N\),
\[
r(z)\le g_0(a_n)
\qquad\text{and}\qquad
r(z)\le g_0(b_n).
\]

Since \(a_n,b_n\le\topel_T\), monotonicity of \(h\) and \(r\) yields
\[
r(h(a_n))\le r(z)
\qquad\text{and}\qquad
r(h(b_n))\le r(z).
\]
Consequently, for every \(n\ge N\),
\[
r(h(a_n))\le g_0(a_n)
\qquad\text{and}\qquad
r(h(b_n))\le g_0(b_n).
\]

Finally,
\[
r(h(\topel_T))
=
r(z)
\le
u
\le
g_0(\topel_T).
\]
Thus \(g_0\) simultaneously controls the greatest element and both
branches beyond the finite level \(N\).
\end{proof}

\begin{lemma}
\label{lemma:composition-waybelow}
Let \(L\) be a continuous dcpo, let \(r:L\to L\) be a
Scott-continuous map that is finitely separated from \(\id_L\), and
let \(h:T\to L\) be Scott-continuous. Then
\(
r\circ h\below h
\)
in the Scott function space \([T\to L]\).
\end{lemma}

\begin{proof}
Let \(\mathcal D\subseteq[T\to L]\) be directed and suppose that
\[
h\le\bigvee\mathcal D.
\]
We must show that \(r\circ h\le g\) for some \(g\in\mathcal D\).

By Lemma~\ref{lemma:two-branch-control}, there exist
\(g_0\in\mathcal D\) and \(N\in\mathbb N\) such that
\[
r(h(\topel_T))\le g_0(\topel_T),
\]
and, for every \(n\ge N\),
\[
r(h(a_n))\le g_0(a_n)
\qquad\text{and}\qquad
r(h(b_n))\le g_0(b_n).
\]

It remains to control only the finite set
\[
F_N
:=
\{\bottom_T\}
\cup
\{\,a_n,b_n:0\le n<N\,\}.
\]
For each \(t\in F_N\), Lemma~\ref{lemma:pointwise-waybelow} gives
\[
r(h(t))\below h(t).
\]
Since
\(
h(t)
\le
\bigvee_{g\in\mathcal D}g(t),
\)
there exists \(g_t\in\mathcal D\) such that
\(
r(h(t))\le g_t(t).
\)

The family \(\mathcal D\) is directed and \(F_N\) is finite. Hence
there exists \(g\in\mathcal D\) such that
\[
g_0\le g
\qquad\text{and}\qquad
g_t\le g
\quad(t\in F_N).
\]

We now verify that \(r\circ h\le g\). If \(t\in F_N\), then
\[
r(h(t))
\le
g_t(t)
\le
g(t).
\]
For every \(n\ge N\), the choice of \(g_0\) gives
\[
r(h(a_n))
\le
g_0(a_n)
\le
g(a_n)
\]
and
\[
r(h(b_n))
\le
g_0(b_n)
\le
g(b_n).
\]
Finally,
\[
r(h(\topel_T))
\le
g_0(\topel_T)
\le
g(\topel_T).
\]

These cases exhaust all elements of \(T\). Therefore
\(
r\circ h\le g.
\)
It follows that \(r\circ h\below h\).
\end{proof}
\subsection{A directed family of way-below approximants}
\label{subsec:directed-waybelow-family}

\begin{theorem}
\label{thm:directed-waybelow-family}
Let \(L\) be an FS-domain, and let
\(
(r_i,M_i)_{i\in I}
\)
be a finitely separating approximate identity on \(L\).
Then, for every Scott-continuous map
\(
h:T\to L,
\)
the family
\(
(r_i\circ h)_{i\in I}
\)
is directed,
every map \(r_i\circ h\) satisfies
\[
r_i\circ h\below h,
\text{ and }
h
=
\bigvee_{i\in I}
r_i\circ h.
\]
\end{theorem}

\begin{proof}
By Lemma~\ref{lemma:composition-waybelow},
\(
r_i\circ h
\below
h
\)
for every \(i\in I\).
Since
\((r_i)_{i\in I}\)
is directed, composition with the Scott-continuous map \(h\) preserves
the pointwise order. Hence
\(
(r_i\circ h)_{i\in I}
\)
is directed.

Finally, for every \(t\in T\),
\[
\bigvee_{i\in I} r_i(h(t))=h(t),
\]
since \((r_i,M_i)_{i\in I}\) is a finitely separating approximate identity on \(L\).
Therefore
\[
\begin{aligned}
\left(
\bigvee_{i\in I}
r_i\circ h
\right)(t)
&=
\bigvee_{i\in I}
r_i(h(t))        \\
&=
h(t).
\end{aligned}
\]
Since directed suprema in
\([T\to L]\)
are computed pointwise,
\[
h
=
\bigvee_{i\in I}
r_i\circ h.
\]
\end{proof}

\subsection{Continuity of the function space}

We are now ready to establish the continuity of the Scott function
space. The preceding subsection shows that every Scott-continuous map
admits a directed family of approximants obtained by composing it with
the finitely separating approximate identity of the codomain. Since
each approximant lies way below the original map, continuity follows
immediately from the definition.

\begin{theorem}
\label{thm:function-space-continuous}
Let \(T\) be Plotkin's tie and let \(L\) be an FS-domain.
Then the Scott function space
\[
[T\to L]
\]
is a continuous dcpo.
\end{theorem}

\begin{proof}
Let \(h\in[T\to L]\). By
Theorem~\ref{thm:directed-waybelow-family},
\[
h=\bigvee_{i\in I}(r_i\circ h),
\]
where the family \((r_i\circ h)_{i\in I}\) is directed and each
\(r_i\circ h\below h\). Thus \(h\) is the directed supremum of
elements way below it. Since \(h\) was arbitrary, every element of
\([T\to L]\) has a directed family of way-below approximants whose
supremum is itself. Hence \([T\to L]\) is a continuous dcpo.
\end{proof}

The preceding theorem shows that Plotkin's tie, although neither an
FS-domain nor an RB-domain, is sufficiently well behaved as a source
space to preserve continuity of Scott function spaces with FS-domain
codomains. In the next section, we show that this phenomenon is
strictly weaker than preserving the FS property itself.

\section{The Closed-Disk--Plotkin-Tie Contrast}
\label{sec:disk-tie-contrast}

The preceding section shows that Plotkin's tie is an unexpectedly
well-behaved source domain for Scott function spaces with FS-domain
codomains. In particular, if \(L\) is an FS-domain, then
\([T\to L]\) is always a continuous dcpo.

A natural question is whether this conclusion can be strengthened.
Must the Scott function space itself be an FS-domain whenever the
codomain is? The answer is negative. In this section we return to the
planar closed-disk domain introduced in
Section~\ref{subsec:closed-disc} and use it to construct a concrete
counterexample.

The resulting example demonstrates that the preservation of continuity
established in Section~\ref{sec:tie-source} is strictly weaker than the
preservation of the FS property. Thus the approximation mechanism
developed in this paper captures precisely what is needed for
continuity, but does not force the existence of a finitely separating
approximate identity in the function space itself.

\subsection{Constructing the counterexample}

We begin with the planar closed-disk domain
\(\Disk\) introduced in
Section~\ref{subsec:closed-disc}. By
Proposition~\ref{prop:closed-disk-properties},
\(\Disk\) is an FS-domain but not an RB-domain.

To apply the results of the preceding section, we adjoin a greatest
element to obtain a pointed FS-domain. We first establish the basic
properties of this extension before considering the associated Scott
function space.

\begin{definition}
\label{def:disk-top}
Let \(\Disk\) denote the planar closed-disk domain introduced in
Definition~\ref{def:closed-disk}. Define
\[
\Disk^{\top}
:=
\Disk\cup\{\topel\},
\]
where \(\topel\notin\Disk\) is a new element satisfying
\[
x\le\topel
\qquad
\text{for every }x\in\Disk.
\]
Thus \(\Disk^{\top}\) is obtained from \(\Disk\) by adjoining a greatest
element.
\end{definition}

The introduction of a greatest element serves two purposes. First, it
produces a pointed FS-domain to which the results of
Section~\ref{sec:tie-source} apply directly. Second, the additional
top element simplifies the construction of the Scott function space
used in the counterexample developed below.

\begin{proposition}
\label{prop:disk-top-properties}
The extension \(\Disk^{\top}\) is a dcpo whose greatest element
\(\topel\) is compact. Moreover, if
\((f_i,M_i)_{i\in I}\) is a finitely separating approximate identity
on \(\Disk\), then the maps
\[
\widehat f_i(x)
:=
\begin{cases}
f_i(x), & x\in\Disk,\\
\topel, & x=\topel,
\end{cases}
\]
together with the finite sets \(M_i\cup\{\topel\}\), form a finitely
separating approximate identity on \(\Disk^{\top}\). Consequently,
\(\Disk^{\top}\) is an FS-domain.
\end{proposition}

\begin{proof}
Let \(D\subseteq\Disk^{\top}\) be directed. If \(\topel\in D\), then
\(\bigvee D=\topel\). Otherwise \(D\subseteq\Disk\), and its supremum
in \(\Disk\) is also its supremum in \(\Disk^{\top}\). Thus
\(\Disk^{\top}\) is a dcpo.

To see that \(\topel\) is compact, suppose that \(D\subseteq
\Disk^{\top}\) is directed and \(\topel\le\bigvee D\). Then
\(\bigvee D=\topel\). If \(\topel\notin D\), then \(D\subseteq\Disk\),
so its supremum belongs to \(\Disk\), contradicting
\(\bigvee D=\topel\). Hence \(\topel\in D\), and therefore
\(\topel\below\topel\).

For each \(i\in I\), the map \(\widehat f_i\) is monotone. We verify
that it preserves directed suprema. Let \(D\subseteq\Disk^{\top}\) be
directed. If \(\bigvee D\in\Disk\), then \(D\subseteq\Disk\), and Scott
continuity of \(f_i\) gives
\[
\widehat f_i\left(\bigvee D\right)
=
f_i\left(\bigvee D\right)
=
\bigvee_{x\in D}f_i(x)
=
\bigvee_{x\in D}\widehat f_i(x).
\]
If \(\bigvee D=\topel\), compactness of \(\topel\) implies that
\(\topel\in D\). Hence both sides of the corresponding equality are
\(\topel\). Thus \(\widehat f_i\) is Scott-continuous.

Since \(f_i\) is finitely separated from \(\id_{\Disk}\) by \(M_i\),
for every \(x\in\Disk\) there exists \(m\in M_i\) such that
\[
\widehat f_i(x)=f_i(x)\le m\le x.
\]
For \(x=\topel\), we have
\(\widehat f_i(\topel)=\topel\le\topel\le\topel\). Therefore
\(\widehat f_i\) is finitely separated from
\(\id_{\Disk^{\top}}\) by \(M_i\cup\{\topel\}\).

The family \((\widehat f_i)_{i\in I}\) is directed because
\((f_i)_{i\in I}\) is directed. Finally, for every \(x\in\Disk\),
\[
\bigvee_{i\in I}\widehat f_i(x)
=
\bigvee_{i\in I}f_i(x)
=
x,
\]
while \(\bigvee_{i\in I}\widehat f_i(\topel)=\topel\). Hence
\(\bigvee_{i\in I}\widehat f_i=\id_{\Disk^{\top}}\), so
\((\widehat f_i,M_i\cup\{\topel\})_{i\in I}\) is a finitely separating
approximate identity on \(\Disk^{\top}\). Therefore
\(\Disk^{\top}\) is an FS-domain.
\end{proof}

\subsection{Continuity of the Scott function space}

The preceding proposition places \(\Disk^{\top}\) within the scope of
Theorem~\ref{thm:function-space-continuous}. We therefore obtain the
following consequence immediately.

\begin{corollary}
\label{cor:tie-disk-top-continuous}
The Scott function space
\[
[T\to\Disk^{\top}]
\]
is a continuous dcpo.
\end{corollary}

\begin{proof}
By Proposition~\ref{prop:disk-top-properties},
\(\Disk^{\top}\) is an FS-domain. The conclusion follows directly from
Theorem~\ref{thm:function-space-continuous}.
\end{proof}

\subsection{Failure of the FS property}

The continuity result established in Corollary~\ref{cor:tie-disk-top-continuous}
cannot, in general, be strengthened to preservation of the FS property.
In fact, the obstruction is completely independent of the closed-disk
domain. It arises from the fact that Plotkin's tie occurs as a
Scott-continuous retract of every Scott function space having \(T\) as
its codomain.

\begin{theorem}
\label{thm:function-space-not-fs}
Let \(X\) be a nonempty FS-domain. Then the Scott function space
\[
[X\to T]
\]
is not an FS-domain.
\end{theorem}

\begin{proof}
Choose an element \(x_0\in X\). Define
\[
c:T\to[X\to T],
\qquad
c(t)(x):=t
\]
for every \(t\in T\) and \(x\in X\). Thus \(c(t)\) is the constant map
with value \(t\). Since constant maps preserve directed suprema, \(c\)
is Scott-continuous.

Next define the evaluation map
\[
\operatorname{ev}_{x_0}:[X\to T]\to T,
\qquad
\operatorname{ev}_{x_0}(h):=h(x_0).
\]
Directed suprema in Scott function spaces are computed pointwise, so
\(\operatorname{ev}_{x_0}\) is Scott-continuous.

For every \(t\in T\),
\[
(\operatorname{ev}_{x_0}\circ c)(t)
=
\operatorname{ev}_{x_0}(c(t))
=
c(t)(x_0)
=
t.
\]
Hence
\[
\operatorname{ev}_{x_0}\circ c
=
\id_T,
\]
showing that \(T\) is a Scott-continuous retract of
\([X\to T]\).

Suppose that \([X\to T]\) were an FS-domain. Since the class of
FS-domains is closed under Scott-continuous retracts, it would follow
that \(T\) is also an FS-domain. This contradicts
Proposition~\ref{prop:plotkin-tie-properties}. Therefore
\([X\to T]\) is not an FS-domain.
\end{proof}

\begin{corollary}
\label{cor:disk-top-tie-not-fs}
The Scott function space
\[
[\Disk^{\top}\to T]
\]
is a continuous dcpo but is not an FS-domain.
\end{corollary}

\begin{proof}
Proposition~\ref{prop:disk-top-properties} shows that
\(\Disk^{\top}\) is an FS-domain. Hence
\([\Disk^{\top}\to T]\) is a continuous dcpo by
Theorem~\ref{thm:tie-target-continuous}
and is not an FS-domain by
Theorem~\ref{thm:function-space-not-fs}.
\end{proof}

\subsection{The Closed-Disk--Plotkin-Tie Contrast}

The preceding results reveal a clear distinction between continuity and
the FS property for Scott function spaces. Although
\(\Disk^{\top}\) is an FS-domain, the Scott function space
\[
[\Disk^{\top}\to T]
\]
is a continuous dcpo but is not an FS-domain. Thus, the continuity
established in Theorem~\ref{thm:tie-target-continuous} cannot, in general, be
strengthened to preservation of the FS property.

This contrast completes the picture developed throughout the paper.
Section~\ref{sec:tie-target} established that Scott function spaces of
the form \([X\to T]\) are continuous whenever \(X\) is an FS-domain,
while Section~\ref{sec:tie-source} established that Scott function
spaces of the form \([T\to L]\) are continuous whenever \(L\) is an
FS-domain. These two positive results show that Plotkin's tie enjoys
remarkable continuity-preserving properties, whether it appears as the
target or as the source of Scott-continuous maps.

The present section demonstrates, however, that continuity is the
strongest conclusion that can be expected in general. Even when the
source is the Lawson closed-disk domain, one of the most natural and
well-studied examples of an FS-domain, the resulting Scott function
space need not itself be an FS-domain.

Consequently, the approximation mechanism developed in this paper
identifies the precise structural ingredient required for continuity of
Scott function spaces, but no more. The Lawson closed-disk domain and
Plotkin's tie together provide a concrete and conceptually transparent
witness that continuity and the FS property are genuinely distinct
notions in the theory of Scott function spaces.

\section{Consequences, Boundaries, and Open Questions}
\label{sec:open-questions}

The principal contribution of this paper is to establish two
complementary continuity theorems for Scott function spaces involving
Plotkin's tie. Together with the Lawson closed-disk example, these
results clarify the relationship between continuity and finite
separation, and identify a precise boundary beyond which the present
approximation methods no longer extend.

Rather than concluding with a summary, we close by discussing the
mathematical consequences of the present work, the limitations of the
methods developed herein, and several natural directions for future
research.

\subsection{Mathematical consequences}

The results obtained in this paper establish that finite separation
provides the approximation required to guarantee continuity of Scott
function spaces in two complementary situations. When Plotkin's tie
appears as the target, finite-layer truncation maps produce directed
families of way-below approximants. When Plotkin's tie appears as the
source, finitely separating approximate identities on the codomain
generate the required approximants by composition.

These two arguments are structurally different. In the first, the
approximating maps arise from the finite-level geometry of the target.
In the second, they arise from the finite-separation structure of the
codomain together with the two-branch geometry of the source. Their
common effect is nevertheless the same: every element of the relevant
Scott function space is recovered as the directed supremum of elements
way below it.

The closed-disk--Plotkin-tie example shows, however, that continuity is
strictly weaker than preservation of the FS property. In particular,
although \(\Disk^{\top}\) is an FS-domain, the Scott function space
\([\Disk^{\top}\to T]\) is continuous but is not an FS-domain.
Consequently, the existence of sufficiently many way-below
approximants in a Scott function space does not imply that those
approximants can be generated by a finitely separating approximate
identity on the function space itself.

Continuity and finite separation should therefore be regarded as
genuinely different approximation phenomena. The results of this paper
show that finite separation may induce continuity under function-space
formation without being preserved by that formation.

\subsection{Boundaries of the present methods}

The proofs developed in this paper rely on two fundamentally different
approximation mechanisms.

\begin{enumerate}
\item
Finite-layer truncation maps arising from the order structure of
Plotkin's tie when it appears as the target.

\item
Finitely separating approximate identities arising from the
FS-structure of the codomain when Plotkin's tie appears as the source.
\end{enumerate}

The first mechanism depends strongly on the fact that the infinite
structure of Plotkin's tie is exhausted by finitely many initial levels
together with two convergent branches. The second depends on the same
two-branch structure in a different way: it allows infinitely many
pointwise way-below estimates to be reduced to finitely many initial
conditions and two controlled tails.

Neither argument extends automatically to an arbitrary continuous
domain. A general target need not admit finite-level truncation maps
with the required way-below properties, while a general source may
contain infinitely many independent directions that cannot be
controlled by finitely many pointwise estimates. Similarly, continuity
of a codomain alone does not provide the finitely separating maps used
in the second argument.

The present methods should therefore be viewed as identifying two
successful approximation principles rather than as providing a general
function-space theorem. They reveal sufficient structural conditions
for continuity, but they do not yet characterize those conditions.

The negative result of the preceding section marks a further boundary.
Even when one of these mechanisms succeeds in establishing continuity,
it need not produce a finitely separating approximate identity on the
resulting function space. Thus the passage from local or externally
supplied approximants to an intrinsic FS-structure remains a genuinely
stronger requirement.

\subsection{Open questions}

The results of this paper suggest several natural problems.

\begin{question}
\label{question:target-characterization}
Characterize those continuous domains \(X\) for which the Scott
function space
\[
[X\to T]
\]
is continuous.
\end{question}

Theorem~\ref{thm:tie-target-continuous} gives the FS property of \(X\) as a
sufficient condition. It is not clear whether this hypothesis can be
weakened, nor which intrinsic property of \(X\) is actually detected
by the finite-layer truncation argument.

\begin{question}
\label{question:source-characterization}
Characterize those continuous domains \(L\) for which the Scott
function space
\[
[T\to L]
\]
is continuous.
\end{question}

Theorem~\ref{thm:function-space-continuous} gives a positive answer
when \(L\) is an FS-domain. The proof uses finite separation to obtain
pointwise way-below estimates and the geometry of \(T\) to promote them
to function-space estimates. It remains open whether finite separation
is necessary, or whether a substantially weaker approximation property
of \(L\) suffices.

\begin{question}
\label{question:intermediate-property}
Is there an intrinsic approximation property, strictly weaker than the
FS property but stronger than continuity, that characterizes the
continuity-preserving behaviour exhibited by the domains considered in
this paper?
\end{question}

Such a property would need to explain how a domain can supply enough
approximants to make a Scott function space continuous without forcing
the function space itself to admit a finitely separating approximate
identity.

\begin{conjecture}
\label{conj:unified-approximation-principle}
There exists an intrinsic approximation principle for continuous
domains that simultaneously generalizes finite-layer truncation maps
and finitely separating approximate identities, and whose presence
characterizes continuity of the corresponding Scott function spaces.
\end{conjecture}

The conjectured principle should not merely assert the existence of
way-below approximants pointwise. It should encode a method for
constructing, from the order-theoretic structure of the source and
target, a directed family of Scott-continuous maps lying way below a
given function and having that function as supremum.

In the target case, the principle should recover finite-layer
truncations of Plotkin's tie. In the source case, it should recover
composition with finitely separating approximating maps, together with
a finite-control argument for the geometry of the source. A successful
formulation would place the two main continuity theorems of this paper
within a single approximation theory.

\subsection*{Final remarks}

The results presented here show that Plotkin's tie occupies a remarkable
position in the theory of Scott function spaces. Whether appearing as
the source or as the target, it admits approximation mechanisms
sufficient to guarantee continuity under one-sided FS assumptions. At
the same time, the Lawson closed-disk example demonstrates that these
mechanisms reach a natural boundary: continuity can be preserved
without preserving the FS property.

The two continuity proofs therefore appear to be manifestations of a
more general phenomenon. In each case, the essential task is not to
construct an FS-structure on the function space, but to construct enough
Scott-continuous approximants below each function to recover it as a
directed supremum. The origin and form of those approximants differ,
but their function-space role is identical.

This suggests that the central problem is no longer simply whether
FS-domains are preserved by Scott function-space formation. The deeper
question is to identify the order-theoretic approximation structures
that are precisely strong enough to ensure continuity of Scott function
spaces. Conjecture~\ref{conj:unified-approximation-principle} formulates
this possibility and points toward a broader theory in which the
results of the present paper would arise as two complementary special
cases.


\end{document}